\documentclass[preprint]{elsarticle}%review
\usepackage{lineno}
\journal{}

\usepackage{stmaryrd}
\usepackage{slashbox}
\usepackage{amsmath}
\usepackage{amssymb}
\usepackage{amsthm}
\usepackage{graphicx}
\usepackage{epic,eepic,epsfig}
\usepackage{color}
\usepackage{subfigure}  % add 1.22
\usepackage{placeins}   % add 1.22
\usepackage{multirow}
\usepackage{epstopdf}
\usepackage{varwidth}
\usepackage{float}
\usepackage[table]{xcolor}

\usepackage{algorithm}
\usepackage{algorithmic}

\newtheorem{theorem}{Theorem}[section]
\newtheorem{example}{Example}[section]
\newtheorem{lemma}{Lemma}[section]
\newtheorem{remark}{Remark}[section]

\usepackage{lineno} % for review purpose

\newcommand{\BLbrace}{ \left\{\kern -0.36em\left\{     }
\newcommand{\BRbrace}{ \right\}\kern -0.36em\right\} }
\newcommand{\Lbrace}{ \left\{\kern -0.23em\left\{     }
\newcommand{\Rbrace}{ \right\}\kern -0.23em\right\} }

\newcommand{\BLbracket }{ \left[\kern -0.19em\left[    }
\newcommand{\BRbracket }{ \right] \kern-0.19em\right] }
\newcommand{\Lbracket }{ \left[\kern -0.09em\left[    }
\newcommand{\Rbracket }{ \right] \kern-0.09em\right] }

\allowdisplaybreaks[4]

\begin{document}

\begin{frontmatter}
	
\title{A Posterior Error Estimator for Mixed Interior Penalty Discontinuous Galerkin Finite Element Method for the  $\boldsymbol{H}(\mathrm{curl})$-Elliptic Problems}

\author[SCNU]{Ming Tang}
\ead{mingtang@m.scnu.edu.cn}

\author[SCNU]{Xiaoqing Xing\corref{cor}}
\ead{xingxq@scnu.edu.cn}

\author[SCNU]{Liuqiang Zhong}
\ead{zhong@scnu.edu.cn}

\cortext[cor]{Corresponding author}
\address[SCNU]{School of Mathematical Sciences, South China Normal University, Guangzhou 510631, China}

\begin{abstract}
In this paper, we  design the first residual type a posteriori error estimator for mixed interior penalty discontinuous Galerkin method for the $\boldsymbol{H}(\mathrm{curl})$-elliptic problems. Then we prove that our residual based a posteriori error indicator is both reliable and efficient. At last, we present some numerical experiments to validate the performance of the indicator within an adaptive mesh refinement procedure.
\end{abstract}

\begin{keyword}
$\boldsymbol{H}(\mathrm{curl})$-elliptic problems, mixed interior penalty discontinuous Galerkin method, a posterior error estimator, reliability, efficiency.
\end{keyword}

\end{frontmatter}
%\maketitle

%\linenumbers

\section{Introduction} 
 In this work, we consider the  $\boldsymbol{H}(\mathrm{curl})$-elliptic problems as follows: find the electric or magnetic field $\boldsymbol{u}$ satisfy
 \begin{eqnarray}
 \label{Equ:1.1}
 \boldsymbol{curl}(\alpha \mathrm{curl}~\boldsymbol{u}) + \beta\boldsymbol{u}=\boldsymbol{f}, &&\ \mbox{in}\ \Omega,\\ \label{Equ:1.2}
 \boldsymbol{u}\cdot\boldsymbol{t} =0, && \mbox{on}\ \partial \Omega,
 \end{eqnarray}
 where $\Omega\subset \mathbb{R}^2$ be a simply connected bounded Lipschitz polygon with boundary $\partial\Omega$ and is partitioned into non-overlapping subdomains $\Omega_i$, $1\leq i \leq m$, $\boldsymbol{f}$ is a given vector field depending on a given external source field, $\boldsymbol{t}$ is the unit tangent on $\partial\Omega$ oriented counter-clockwisely, $\alpha\geq\alpha_0>0$ and $\beta\geq\beta_0>0$ are piecewise constants in $\Omega_i$, $\alpha_0$ and $\beta_0$ are constans.
 We recall that, $\mathrm{curl}\  \boldsymbol{v}={\partial v_2}/{\partial x}-{\partial v_1}/{\partial y}$ for a vector field $\boldsymbol{v}=(v_1, v_2)$, while $\boldsymbol{curl}\ \phi =({\partial \phi}/{\partial y}, -{\partial \phi}/{\partial x})$ for a scalar function $\phi$.
 Our numerical scheme and the a posteriori error analysis are based on
 a mixed formulation of \eqref{Equ:1.1}-\eqref{Equ:1.2}, which is obtained by introducing an auxiliary variable $p = \mathrm{curl}~\boldsymbol{u}$
 \begin{eqnarray} \label{Eqn:hhcurl}
 \boldsymbol{curl}~(\alpha p) + \beta\boldsymbol{u} = \boldsymbol{f},&& \mbox{in}\ \Omega,
 \\ \label{Eqn:hq} 
 p - \mathrm{curl}~\boldsymbol{u} = 0, && \mbox{in}\ \Omega,
 \\ \label{Eqn:hboundary}
 \boldsymbol{u} \cdot \boldsymbol{t} = 0,&& \mbox{on}\ \partial\Omega.
 \end{eqnarray}
 
 Discontinuous Galerkin (DG) finite element method is one of popular methods for numerical solution of partial differential equations.
 Compared with the traditional conforming finite element method, the DG finite element method has advantages as follows: to allow incompatible with suspension point grid,  to deal with complex boundary and interface problems easily, and to implement partial encryption and each unit of polynomial independent selection easily.
 One of the key features of the DG method is that the discontinuous approximation at element interfaces naturally allows jump discontinuities in the solution if element boundaries are placed along them \cite{KoprivaGassner21:1}.
 DG method has been developed to solve many equations, such as elliptic problems \cite{BrezziManzini00:365}, parabolic equations \cite{Riviere08Book}, advection-diffusion-reaction problems \cite{HoustonSchwab02:2133}.
 The DG methods include locally DG(LDG) method \cite{Castillo02:524}, interior penalty DG(IPDG) method \cite{Arnold82:742}.
 The discontinuous finite element method for $\boldsymbol{H}(\mathrm{curl})-$elliptic problems is still in its infancy.
 Chung and Kim \cite{ChungKim14:1} proposed an improved Feti-DP algorithm and convergence analysis for the mixed interleaved discontinuous finite element method for the two-dimensional $\boldsymbol{H}(\mathrm{curl})-$elliptic problems.
 
 On the other hand,
 in practical engineering applications and scientific calculations, there are many factors that may cause strong singularities in the propagation of electromagnetic fields. For example, the material coefficient of the medium in the electromagnetic wave propagation area is discontinuous, or the source term of the generated electromagnetic field is not smooth \cite{CostabelDauge00:221,CostabelDauge03:807}. Although these singularities can be overcomed by uniformly densifying the grid when performing numerical solutions, consistent densification can lead to a sharp increase in computational cost. Hence, adaptive finite element emerges as the times require.
 In the past few decades, adaptive finite element method have been proven to be a useful and effective tool in scientific computing. The standard adaptative process is as follows
 $
 \text { SOLVE } \rightarrow \text { ESTIMATE } \rightarrow \text { MARK } \rightarrow \text { REFINE. }
 $
 The adaptive finite element method is based on a posteriori error estimation.
 It automatically refines and optimizes mesh generation according to the local posteriori error indicator on the element.
 It is a numerical calculation method with high reliability and efficiency.
 
 Most of the work on the convergence of the adaptive method for the $\boldsymbol{H}(\mathrm{curl})$-elliptic equations focuses on the edge finite element.
 For example,
 using the so-called interior node property and oscillation marker as technical assumptions, the convergence of the lowest order edge elements of the  N\'{e}d\'{e}lec's first family of adaptive  for two-dimensional and three-dimensional eddy current equations are proved in [4,14]%\cite{CarstensenHoppe05:19,HoppeSchoberl09:657}
 , respectively.
 Chen, Xu and Zou \cite{ChenJQXuYF09:2950} proved that an adaptive method for three dimensional static Maxwell equations without additional marking of oscillation terms and gives corresponding proof of convergence with the lowest order edge elements of  N\'{e}d\'{e}lec's first family.
 Zhong, Shu, Chen and Xu \cite{Zhong10AEFEM}
 proved that the three-dimensional $\boldsymbol{H}(\mathrm{curl})$-elliptic problem with variable coefficients is convergent by using high order and the two family of N\'{e}d\'{e}lec edge elements.
 There are also some studies on the posteriori error estimator of the adaptive DG finite element method for $\boldsymbol{H}(\mathrm{curl})$-elliptic problem \cite{HoustonPerugia07:122,Xingzhong12:18Eng}.
 Houston, Perugia and Schotzau \cite{HoustonPerugia07:122} gave the residual-type posteriori error estimator and proved the reliability and efficiency of the error estimator. Xing and Zhong \cite{Xingzhong12:18Eng} gave a simplified posteriori error indicator and proved corresponding upper bound.
 Recently, Zhong, Chen and Xing \cite{ZhongLQChengTEng16:92} proved the convergence of the adaptive interior penalty DG methods.
 
 Meanwhile,
 there are many successful works of solving the Maxwell's equations by the mixed finite element method, e.g. [18,19,21,15]. %\cite{JiangWLiuN14:159,JogNandy14:887,LiuTobon15:458,HoustonPerugia05:727}.
 However,
 For adaptive mixed finite element method solving Maxwell's equations, there are only few research results for a posterior error estimator.
 For example,
 Carstensen, Hoppe, Sharma and Warburton \cite{CarstensenHoppepe11:13} studied a posteriori error estimation of the hybridized finite element method and proved the reliability of the estimator up to a consistency error.
 Chung, Yuen and Zhong \cite{ChungYuen14:613} studied a posteriori error estimation of the  staggered discontinuous Galerkin method for time-harmonic Maxwell's equations and proved that residual based a posteriori error indicator is both reliable and efficient.
 %Other results can be seen \cite{ChungEngquist09:3820,ChungLee12:1241}.	
 As far as we know, there are not any published literatures on the posteriori error estimation of the adaptive mixed finite element method for $\boldsymbol{H}(\mathrm{curl})-$elliptic problems \eqref{Equ:1.1}-\eqref{Equ:1.2}. The main idea of the manuscript comes from \cite{ChungYuen14:613}.
 However, one of main tool, a Cl\'{e}ment-type quasi-interpolation operator given by \cite{Schoberl08:633},
 can not be used for 2D finite element space.
 Here, we use the Helmholtz decomposition and operators in articles \cite{ScottZhang90:483,CarstensenHoppe05:19}
 for estimation.
 
 Here is some notation used throughout the paper. The following shorthand notation will be used to avoid the repeated constants, following \cite{XuJC92:581},
 $x\lesssim y$ and $x\approx y$ means $x \leq C_1y$ and $C_2x\leq y\leq C_3x$,  where $C_1$, $C_2$ and $C_3$ are generic positive constants.
 
 The rest of the article is organized as follows.
 In Section \ref{sec:2}, we introduce some basic notations, present the variational form of the model problem \eqref{Equ:1.1}-\eqref{Equ:1.2}, and design a residual type a posteriori error estimator.
 In Section \ref{sec:4} and Section \ref{sec:5}, we show that this indicator is reliable and effective, respectively.
 In Section \ref{sec:6},  we report some numerical results in support of theoretical results.
 
 %%%%%%%%%%%%% Section 2 %%%%%%%%%%

 \section{Mixed IPDG  method and a posteriori error indicator}\label{sec:2}
 In this section, we give the continuous variational problem, the discrete variational problem of mixed IPDG method, and the definition of the a posteriori error indicator.
 
 \subsection{Continuous variational problem}
 
 For any domain $D \subset \mathbb{R}^{2}$, we use standard definitions for the Sobolev spaces $H^{s}(D)$ and $\boldsymbol{H}^{s}(D)$ of scalar and vector-valued square integrable functions with inner products $(\cdot,\cdot)_{s,D}$ and  associated norms $\|\cdot\|_{s, D}$ for $s \geq 0$, respectively.
 We refer to $L^{2}(D)$ and $\mathbf{L}^{2}(D)$ as the Hilbert spaces of scalar and vector-valued square integrable functions with inner products $(\cdot, \cdot)_{0, D}$ and associated norms $\|\cdot\|_{0, D}$, respectively.
 For simplicity, we drop the subscript when $G = D$. Then, the spaces are defined by
 \begin{eqnarray*}
 	&\boldsymbol{H}(\mathrm{curl}, \Omega):=\left\{\boldsymbol{v}: \boldsymbol{v} \in\boldsymbol{L}^{2}(\Omega), \mathrm{curl}~ \boldsymbol{v} \in L^{2}(\Omega)\right\},\\
 	& \boldsymbol{H}_0(\mathrm{curl}, \Omega) := \left\{\boldsymbol{v}: \boldsymbol{v} \in\boldsymbol{H}(\mathrm{curl},\Omega), \boldsymbol{v}\cdot\boldsymbol{t} = 0~on~\partial\Omega \right\}.
 \end{eqnarray*}
 The space $\boldsymbol{H}(\mathrm{curl},\Omega)$ is equipped with norm $\|\boldsymbol{v}\|^2_{\mathrm{curl},\Omega} := \|\boldsymbol{v}\|^2_{0,\Omega} + \|\mathrm{curl}~\boldsymbol{v}\|^2_{0,\Omega}$ for any $\boldsymbol{v}\in \boldsymbol{H}(\mathrm{curl},\Omega)$.
 
 We simplify the symbols $\boldsymbol{H}_0(\mathrm{curl},\Omega)$ and $L^2(\Omega)$ to $\boldsymbol{U}$ and $\mathbb{Q}$, respectively.
 In this manuscript, we assume that $\boldsymbol{f}\in\boldsymbol{H}(\operatorname{div}, \Omega)=\{\boldsymbol{v}: \boldsymbol{v} \in \boldsymbol{L}^{2}(\Omega), \nabla \cdot \boldsymbol{v} \in L^{2}(\Omega)\}$.
 The variational form for \eqref{Eqn:hhcurl}-\eqref{Eqn:hboundary} is to find  $(\boldsymbol{u},p)\in \boldsymbol{U}\times\mathbb{Q}$ such that
 \begin{eqnarray}
 &&a({p}, {q})-b(\boldsymbol{u}, {q})=\ell_{1}({q}), ~~\forall {q} \in \mathbb{Q},\label{Eqn:weak_1} \\
 &&d(\boldsymbol{v}, {p})+c(\boldsymbol{u}, \boldsymbol{v})=\ell_{2}(\boldsymbol{v}),~  \forall \boldsymbol{v} \in \boldsymbol{U},\label{Eqn:weak_2}
 \end{eqnarray}
 where the four bilinear forms given by
 \begin{eqnarray}
 &&a({p}, {q}):=({p}, {q}), \label{Eqn:a}\\
 &&b(\boldsymbol{u}, {q}):=(\mathrm{curl}~ \boldsymbol{u}, {q}), \label{Eqn:b}\\
 &&c(\boldsymbol{u}, \boldsymbol{v}):=(\beta\boldsymbol{u}, \boldsymbol{v}), \label{Eqn:c}\\
 &&d(\boldsymbol{v}, {p}) := (\mathrm{curl}~\boldsymbol{v}, \alpha{p}),\label{Eqn:d}
 \end{eqnarray}
 and two linear functionals $\ell_{1}(\cdot) \in \mathbb{Q}^{*}$, $\ell_{2}(\cdot) \in \boldsymbol{U}^{*}$, where $\mathbb{Q}^{*}$ and $\boldsymbol{U}^{*}$ are the dual spaces of $\mathbb{Q}$ and $\boldsymbol{U}$, respectively,  as follows
 \begin{eqnarray}
 &&\ell_{1}({q}):=0, \label{Eqn:l_1}\\
 &&\ell_{2}(\boldsymbol{v}):=(\boldsymbol{f}, \boldsymbol{v}).\label{Eqn:l_2}
 \end{eqnarray}
 
 In order to prove the well-posedness of continuous variational problem \eqref{Eqn:weak_1}-\eqref{Eqn:weak_2} and the reliability of a posteriori error indicator(see Lemma \ref{Lem:leqell_1ell_2}).
 We also define the operator $\mathcal{A}:(\boldsymbol{U}\times\mathbb{Q}  ) \mapsto(\boldsymbol{U}\times\mathbb{Q}  )^{*}$ by
 $$
 (\mathcal{A}( \boldsymbol{u},p))(\boldsymbol{v},q):=a(p, q)-b(\boldsymbol{u}, q)+d(\boldsymbol{v}, p)+c(\boldsymbol{u}, \boldsymbol{v}), \ \ \text {for all } \boldsymbol{u}, \boldsymbol{v} \in \boldsymbol{U},  p, q \in \mathbb{Q}.
 $$
 Thus, the operator form of the equations \eqref{Eqn:weak_1}-\eqref{Eqn:weak_2} is obtained
 \begin{equation}\label{Eqn:Al}
 (\mathcal{A}(\boldsymbol{u},{p}))(\boldsymbol{v},{q} )=\ell(\boldsymbol{v},{q}),
 \end{equation}
 where $\ell(\boldsymbol{v},{q})=\ell_{2}(\boldsymbol{v})+\ell_{1}({q})$.
 
 The following lemma provides the existence and uniqueness of solutions to the variational problem \eqref{Eqn:weak_1}-\eqref{Eqn:weak_2}.
 
 \begin{lemma}[\cite{ChungYuen14:613}, Lemma 2.1]\label{WVP}
 	Let $\Omega$ be a bounded Lipschitz polygon with connected boundary $\partial\Omega$. Then $\mathcal{A}$ is a  continuous and bijective linear operator. Moreover, for any $\left(\ell_{1}, \ell_{2}\right) \in \mathbb{Q}^{*} \times \boldsymbol{U}^{*}$ given by \eqref{Eqn:dweak_1} and \eqref{Eqn:dweak_2}, respectively, then the system \eqref{Eqn:weak_1}-\eqref{Eqn:weak_2} has a unique solution $(\boldsymbol{u},{p}) \in \boldsymbol{U}\times\mathbb{Q}$ such that
 	\begin{equation}
 	\|(\boldsymbol{u}, {p})\|_{\boldsymbol{U} \times \mathbb{Q}}:=\left(\|\boldsymbol{u}\|_{ {\boldsymbol{U}}}^{2}+\|{p}\|_{\mathbb{Q}}^{2}\right)^{1 / 2} \lesssim
 	\|\ell_{1}\|_{\mathbb{Q}^{*}}
 	+
 	\|\ell_{2}\|_{\boldsymbol{U}^{*}}\label{Eqn:up},
 	\end{equation}
 	where $\|\cdot\|_{\mathbb{Q}^*}$ and $\|\cdot\|_{\boldsymbol{U}^*}$ are dual norms in $\mathbb{Q}^*$ and $\boldsymbol{U}^*$, respectively.
 \end{lemma}
 
 \subsection{Discrete variational problem}
 
 Before presenting the discrete variational problem, we introduce some preliminaries.
 Given a shape-regular triangulation $\mathcal{T}_h$ for $\Omega$. For $\tau\in\mathcal{T}_h$, we write $h_{\tau} = |\tau|^{1/2} $ to denote
 the local mesh size of the element $\tau$, where $|\tau|$ is the Lebesgue measure of $\tau$. Let $h=\max_{\tau\in\mathcal{T}_h} h_{\tau}$.
 
 Let  $\mathcal{E}_h$ be the set of all the edges, $\mathcal{E}_h^0=\mathcal{E}_h\backslash \partial \Omega$ be the set of all the interior edges, and $\mathcal{E}_h^{\partial}=\mathcal{E}_h\cap \partial \Omega$ be the set of all the boundary edges, then $\mathcal{E}_h=\mathcal{E}_h^0\bigcup\mathcal{E}_h^{\partial}$.
 
 For $\mathcal{T}^{\prime}_h\subseteq \mathcal{T}_h$ and $\mathcal{E}_h^{\prime}\subseteq \mathcal{E}_h$,the discrete $L^2$ inner product and norm are given by
 \begin{eqnarray*}
 	& \displaystyle
 	(\boldsymbol{v}, \boldsymbol{w})_{\mathcal{T}^{\prime}_h}
 	=\sum\limits_{\tau \in \mathcal{T}^{\prime}_h}(\boldsymbol{v}, \boldsymbol{w})_{\tau}
 	=\sum\limits_{\tau\in \mathcal{T}^{\prime}_h} \int_{\tau} \boldsymbol{v}\cdot \boldsymbol{w} \mathrm{d} x, \quad\|\boldsymbol{v}\|_{\mathcal{T}^{\prime}_h}^{2}
 	=(\boldsymbol{v}, \boldsymbol{v})_{\mathcal{T}^{\prime}_h},
 	\\
 	& \displaystyle
 	\langle \boldsymbol{v}, \boldsymbol{w} \rangle_{\mathcal{E}_h^{\prime}}
 	=\sum\limits_{e \in \mathcal{E}_h^{\prime}} \langle \boldsymbol{v}, \boldsymbol{w} \rangle_{e}
 	=\sum\limits_{e \in \mathcal{E}_h^{\prime}} \int_{e} \boldsymbol{v}\cdot \boldsymbol{w}\mathrm{d} s, \quad \|\boldsymbol{v}\|_{\mathcal{E}_h^{\prime}}^{2}
 	=\langle \boldsymbol{v}, \boldsymbol{v} \rangle_{\mathcal{E}_h^{\prime}}.
 \end{eqnarray*}
 
 For any $e\in\mathcal{E}_h^{0}$ with $e = \partial\tau_1\cap\partial\tau_2$, we define the average, tangential jump and normal jump for a vector function $\boldsymbol{w}$ by
 \begin{eqnarray*}
 	&	\{\{\boldsymbol{w}\}\}_e = (\boldsymbol{w}|_{\tau_1} + \boldsymbol{w}|_{\tau_2})/{2}, 
 	\\
 	&	[[ \boldsymbol{w}]]_e= \boldsymbol{w}|_{\tau_1}\cdot\boldsymbol{t}_1 + \boldsymbol{w}|_{\tau_2}\cdot\boldsymbol{t}_2, 
 	\\ 
 	&\ [\boldsymbol{w}]_e = \boldsymbol{w}|_{\tau_1}\cdot \boldsymbol{n}_1 + \boldsymbol{w}|_{\tau_2}\cdot \boldsymbol{n}_2,
 \end{eqnarray*}
 where $\boldsymbol{w}|_{\tau_i}$ denotes the value of $\boldsymbol{w}$ on $\tau_i$, $\boldsymbol{t}_i$ and $\boldsymbol{n}_i$ are the unit tangential vectors and the outward unit normal vectors on $e$ for $\tau_i$ ($i=1,2$), respectively.
 
 Similarly, we define the average and the tangential jump on $e$ for a scalar function $\phi$ as
 \begin{equation*}
 \{\{\phi\}\}_{e}=(\left.\phi\right|_{\tau_{1}}+\left.\phi\right|_{\tau_{2}}) / 2, ~[[\phi]]_{e}=\left.\phi\right|_{\tau_{1}} \boldsymbol{t}_{1}+\left.\phi\right|_{\tau_{2}} \boldsymbol{t}_{2},
 \end{equation*}
 where $\phi|_{\tau_i}$ denotes the value of $\phi$ on $\tau_i$, $i=1,2$.
 
 For any $e\in\mathcal{E}_h^{\partial}$, there is a element $\tau\in\mathcal{T}_h$ such that $e\in\partial\tau \cap\partial\Omega$, we define the average,  tangential jump and normal jump for a vector function $\boldsymbol{w}$ are defined as
 \begin{eqnarray*}
 	\{\{ \boldsymbol{w}\}\}_e=\boldsymbol{w}|_{\tau}, \  [[ \boldsymbol{w}]]_e=\boldsymbol{w}|_\tau\cdot\boldsymbol{t},\
 	[\boldsymbol{w}]_e = \boldsymbol{w}|_{\tau}\cdot\boldsymbol{n}
 \end{eqnarray*}
 where $\boldsymbol{w}|_{\tau}$ denotes the value of $\boldsymbol{w}$ on $\tau$ and $\boldsymbol{n}$ denotes the outward unit normal vectors on $e$ for $\tau$.
 
 For a scalar function $\phi$, its average and  tangential jump on $e$ are defined as
 \begin{equation*}
 \{\{\phi\}\}_{e}=\phi|_\tau, ~[[\phi]]_{e}=\phi|_{\tau} \boldsymbol{t},
 \end{equation*}
 where $\phi|_{\tau}$ denote the value of $\phi$ on $\tau$.
 
 The DG methods are based on the approximation of the vector field $\boldsymbol{u}$ and $p$ by elementwise polynomials, thus giving rise to the finite dimensional function spaces
 \begin{eqnarray*}
 	&&\boldsymbol{U}_{h}:=\left\{\boldsymbol{v}_{h} \in \mathbf{L}^{2}(\Omega)\left|~\boldsymbol{v}_{h}\right|_{\tau} \in \mathcal{R}_{1}(\tau),~\boldsymbol{v}_h|_e = 0, \forall \tau \in \mathcal{T}_{h}\right\}, \\
 	&&\mathbb{Q}_{h}:=\left\{{q}_{h} \in {L}^{2}(\Omega)\left|~{q}_{h}\right|_{\tau} \in P_0(\tau), \forall \tau \in \mathcal{T}_{h}\right\},
 \end{eqnarray*}
 where $\mathcal{R}_{1}(\tau)=\left\{\exists \boldsymbol{\alpha} \in \mathbb{R}^{2}, \exists \beta \in \mathbb{R}, \forall \boldsymbol{x}=\left(x_{1}, x_{2}\right) \in \tau: \boldsymbol{q}(\boldsymbol{x})=\boldsymbol{\alpha}+\beta\left(-x_{2}, x_{1}\right)\right\}$,
 and $P_{0}(\tau)$ denotes the constant in $\tau$.
 
 Now, we present the  mixed interior penalty discontinuous Galerkin(MIPDG) finite element method for the system \eqref{Eqn:hhcurl}-\eqref{Eqn:hq}: find $(\boldsymbol{u}_h,p_h)\in \boldsymbol{U}_h\times\mathbb{Q}_h$ such that
 \begin{eqnarray}
 a_h({p}_h, {q}_h)-b_h(\boldsymbol{u}_h, {q}_h)
 &=&
 \ell_{1,h}({q}_h)+d_{1,h}(\boldsymbol{u}_h,q_h),
 \quad  \forall {q}_h \in \mathbb{Q}_h,\label{Eqn:dweak_1} \\
 d_h(\boldsymbol{v}_h, {p}_h)+c_h(\boldsymbol{u}_h, \boldsymbol{v}_h)
 &=&
 \ell_{2,h}(\boldsymbol{v}_h)+d_{2,h}(\boldsymbol{u}_h,\boldsymbol{v}_h),
 \quad \forall \boldsymbol{v}_h \in \boldsymbol{U}_h,\label{Eqn:dweak_2}
 \end{eqnarray}
 where
 \begin{eqnarray}
 &&a_h({p}_h, {q}_h):=({p}_h, {q}_h)_{\mathcal{T}_h}, \label{Eqn:da}\\
 &&b_h(\boldsymbol{u}_h,{q}_h):=(\mathrm{curl}_h~ \boldsymbol{u}_h,{q}_h)_{\mathcal{T}_h}, \label{Eqn:db}\\
 &&c_h(\boldsymbol{u}_h, \boldsymbol{v}_h):=(\beta\boldsymbol{u}_h, \boldsymbol{v}_h)_{\mathcal{T}_h}, \label{Eqn:dc}\\
 &&d_h(\boldsymbol{v}_h,p_h):=(\mathrm{curl}_h\boldsymbol~\boldsymbol{v}_h,\alpha p_h)_{\mathcal{T}_h}, \label{Eqn:dd}\\
 &&\ell_{1,h}({q}_h):=0, \label{Eqn:dl_1}\\
 &&\ell_{2,h}(\boldsymbol{v}_h):=(\boldsymbol{f}, \boldsymbol{v}_h)_{\mathcal{T}_h},\label{Eqn:dl_2}
 \\ &&
 d_{1,h}\left(\boldsymbol{u}_{h}, {q}_{h}\right):=-<\{\{{q}_{h}\}\},[[\boldsymbol{u}_{h}]]>_{\mathcal{E}_{h}}, \\
 &&d_{2,h}\left(\boldsymbol{u}_{h}, \boldsymbol{v}_{h}\right):=<\{\{\alpha\mathrm{curl}_h~ \boldsymbol{u}_{h}\}\}-\kappa h^{-1}_e[[\boldsymbol{u}_{h}]],[[\boldsymbol{v}_{h}]]>_{\mathcal{E}_{h}},
 \end{eqnarray}
 with $\kappa>0$ is a penalty parameter and should be taken large enough.
 \begin{remark}\label{Rem:1}
 	\begin{itemize}
 		\item Comparing with the continuous variational problem \eqref{Eqn:weak_1}-\eqref{Eqn:weak_2} and the discrete variational problem \eqref{Eqn:dweak_1}-\eqref{Eqn:dweak_2}, the definitions of the bilinear terms, which without including $\mathrm{curl}_h$, are the same. In order to be consistent with other symbols, we add the subscript $h$ to the bilinear terms in the discrete variational form.
 		\item The calculation of $\mathrm{curl}_h$ in the bilinear terms of the discrete variational problem is piecewise derivation.
 		\item Compared with the continuous variational form, the discrete variational form adds two terms $d_{1,h}$ and $d_{2,h}$.
 	\end{itemize}
 \end{remark}
 
 In order to give the well-posedness of the discrete variational problems, we need to introduce the suitable IPDG form of the $\boldsymbol{H}(\mathrm{curl})-$elliptic problems: find $\boldsymbol{u}_h\in\boldsymbol{U}_h$, such that
 \begin{equation}\label{DVP:IPDG}
 a_{I P}\left(\boldsymbol{u}_{h}, \boldsymbol{v}_{h}\right)
 =\left(\boldsymbol{f}, \boldsymbol{v}_{h}\right)_{\mathcal{T}_{h}},
 \end{equation}
 where
 \begin{eqnarray}\nonumber
 a_{I P}\left(\boldsymbol{u}_{h}, \boldsymbol{v}_{h}\right)
 &=& \left(\beta\boldsymbol{u}_{h}, \boldsymbol{v}_{h}\right)_{\mathcal{T}_{h}}
 +\left(\alpha\mathrm{curl}~ \boldsymbol{u}_{h}, \mathrm{curl}~ \boldsymbol{v}_{h}\right)_{\mathcal{T}_{h}}
 -<\{\{\mathrm{curl}~ \boldsymbol{v}_{h}\}\} ,[[\alpha\boldsymbol{u}_{h}]]>_{\mathcal{E}_{h}}
 \\ \label{Def:Bilinear:aIP}
 && - <\{\{\alpha\mathrm{curl}~ \boldsymbol{u}_{h}\}\} ,[[\boldsymbol{v}_{h}]]>_{\mathcal{E}_{h}}
 +
 \kappa<h_e^{-1}[[\boldsymbol{u}_{h}]] ,[[\boldsymbol{v}_{h}]]>_{\mathcal{E}_{h}}.
 \end{eqnarray}
 
 \begin{remark}
 	Let $q_h=\alpha\mathrm{curl}~\boldsymbol{v}_h$ in \eqref{Eqn:dweak_1}-\eqref{Eqn:dweak_2}, and  subtract \eqref{Eqn:dweak_1} from \eqref{Eqn:dweak_2} then lead to  \eqref{DVP:IPDG}.
 \end{remark}
 
 To provide the existence and uniqueness of solutions to the variational problem \eqref{DVP:IPDG}, we need to introduce the following norm.
 \begin{eqnarray*}\label{Eqn:fs_h1}
 	|||\boldsymbol{v}_h|||_{h}^2=\|\mathrm{curl}~ \boldsymbol{v}_h\|^2_{\mathcal{T}_h}+\|\boldsymbol{v}_h\|^2_{\mathcal{T}_h}+\kappa\|h_e^{-\frac{1}{2}}[[ \boldsymbol{v}_h]] \|_{\mathcal{E}_h}^2,\quad \forall \boldsymbol{v}_h\in ( H^{1}\left(\mathcal{T}_{h}\right))^2, \kappa>0.
 \end{eqnarray*}
 
 Similar to \cite{BonitoNochetto10:734},
 by using the Cauchy-Schwarz inequality, trace inequality and inverse inequality,  it is easy to verify that $a_h(\cdot,\cdot)$ is bounded by $\||\cdot|\|_h$, i.e.,
 \begin{equation}\label{Eqn:bb}
 a_{h}\left(\boldsymbol{w}_{h}, \boldsymbol{v}_{h}\right)
 \leqslant C\||\boldsymbol{w}_{h}|\|_{h}\||\boldsymbol{v}_{h}|\|_{h}, \quad \forall \boldsymbol{w}_{h}, \boldsymbol{v}_{h} \in \mathbb{V}_{h}.
 \end{equation}
 Furthermore, for the coercivity of the bilinear forms $a_h(\cdot.\cdot)$ on $\mathbb{V}_h$, we have
 \begin{equation}\label{Eqn:cc}
 a_{h}\left(\boldsymbol{v}_{h}, \boldsymbol{v}_{h}\right) \geqslant C\||\boldsymbol{v}_{h}|\|_{h}^{2}, \quad \forall \boldsymbol{v}_{h} \in \mathbb{V}_{h}.
 \end{equation}
 Combining \eqref{Eqn:bb} and \eqref{Eqn:cc}, we obtain the well-posedness of the discrete variational problem \eqref{DVP:IPDG}.
 
 Furthermore, Lemma \ref{Lem:solveeq} shows that the variational problem \eqref{Eqn:dweak_1}-\eqref{Eqn:dweak_2} and the variational problem \eqref{DVP:IPDG} have equivalent form.
 The proof use similar arguments in \cite{CarstensenHoppe09:27}
 and is skipped here.

 \begin{lemma}\label{Lem:solveeq}
 	If $\left(\boldsymbol{u}_{h},{p}_{h} \right) \in\left(\boldsymbol{U}_{h}, \mathbb{Q}_{h}\right)$ is the solution of equation \eqref{Eqn:dweak_1}-\eqref{Eqn:dweak_2}, then $\boldsymbol{u}_{h} \in \boldsymbol{U}_{h}$ is the solution of the variational problem \eqref{DVP:IPDG}.
 	On the contrary,
 	if $\boldsymbol{u}_{{h}} \in \boldsymbol{U}_{h}$ is the solution of the variational problem \eqref{DVP:IPDG}, then there is a corresponding ${p}_{h} \in \mathbb{Q}_{h}$ makes $\left(\boldsymbol{u}_{{h}},{p}_{h}\right) \in\left(\boldsymbol{U}_{h}, \mathbb{Q}_{h}\right)$ is the solution of	\eqref{Eqn:dweak_1}-\eqref{Eqn:dweak_2}.
 \end{lemma}
 
 \subsection{A posteriori error indicator}
 %We shall consider a residual type a posteriori error indicator.
 For any $\tau\in \mathcal{T}_h$, $e\in \mathcal{E}_h$ and $\left(\boldsymbol{v}_{h}, {q}_{h}\right) \in \mathbb{Q}_{h} \times  \boldsymbol{U}_{h}$, we introduce the following element-wise residuals and edge-wise jump residuals as
 $$
 \begin{array}{l}
 R_{1}\left(\boldsymbol{v}_{h}, {q}_{h}\right)|_{\tau}
 :=
 {q}_{h}|_{\tau}-\mathrm{curl}_h\boldsymbol{v}_{h}|_{\tau}, \\
 R_{2}\left(\boldsymbol{v}_{h}, {q}_{h}\right)|_{\tau}
 :=
 \boldsymbol{f}|_{\tau}-\left(\boldsymbol{curl}_h~\alpha {q}_{h}+\beta\boldsymbol{v}_{h}\right)|_{\tau}, \\
 R_{3}\left(\boldsymbol{v}_{h}\right)|_{\tau}
 :=
 \nabla \cdot\left(\boldsymbol{f}-\beta\boldsymbol{v}_{h}\right)|_{\tau}, \\
 J_{1}\left({q}_{h}\right)|_{e}
 :=
 [[\alpha{q}_{h}]]_e, \\
 J_{2}\left(\boldsymbol{v}_{h}\right)|_{e}:=[\boldsymbol{f}-\beta\boldsymbol{v}_{h}]_e,\\
 J_{3}\left(\boldsymbol{v}_{h}\right)|_{e}:=[[\boldsymbol{v}_h]]_e.
 \end{array}
 $$
 The local error estimator on $\tau\in \mathcal{T}_h$ is defined as
 $$
 \begin{aligned}
 \eta^{2}\left(\boldsymbol{v}_{h}, {q}_{h} ; \tau\right)
 :=&
 \|R_{1}\left(\boldsymbol{v}_{h}, {q}_{h}\right)\|_{0,\tau}^{2}
 +
 h_{\tau}^{2}\left(\|R_{2}\left(\boldsymbol{v}_{h}, {q}_{h}\right)\|_{0,\tau}^{2}
 +
 \|R_{3}\left(\boldsymbol{v}_{h}\right)\|_{0,\tau}^{2}\right) \\
 &+
 \sum_{e \in \partial \tau} h_{e}\left(\|J_{1}\left({q}_{h}\right)\|_{0,e}^{2}
 +
 \|J_{2}\left(\boldsymbol{v}_{h}\right)\|_{0,e}^{2}\right)
 +
 \kappa\sum_{e\in\partial\tau}h^{-1}_e\left\|J_3(\boldsymbol{v}_h)\right\|^2_{0, e},
 \end{aligned}
 $$
 where $h_{\tau}$ denotes the diameter of the element $\tau$. The mesh $\mathcal{T}_h$ is shape-regular which implies that $h_{\tau} \approx h_{e}$.
 
 Then the global error estimator on $ \mathcal{T}_h$ is defined as
 \begin{equation}\label{eta}
 \eta^{2}\left(\boldsymbol{v}_{h}, {q}_{h} ; \mathcal{T}_{h}\right)=\sum_{\tau \in \mathcal{T}_{h}} \eta^{2}\left(\boldsymbol{v}_{h}, {q}_{h} ; \tau\right).
 \end{equation}
 
 \section{Reliability analysis}\label{sec:4}
 
 For any  $(\boldsymbol{v},q)\in\boldsymbol{U}\times\mathbb{Q}$ and $(\boldsymbol{v}_h,q_h)\in\boldsymbol{U}_h\times\mathbb{Q}_h$, we define the following error
 \begin{eqnarray}\nonumber
 \| (\boldsymbol{v},q)-(\boldsymbol{v}_h,q_h) \|^2_{DG}
 &:=&
 \|{q-q_h}\|_{0,\Omega}^{2}
 +
 \|\boldsymbol{v}-\boldsymbol{v}_{h}\|_{0,\Omega}^{2}
 +\|\mathrm{curl}_h~(\boldsymbol{v}-\boldsymbol{v}_{h})\|_{0,\Omega}^{2}
 \\ \label{Eqn:DG}
 &&+ \kappa\sum\limits_{e\in \mathcal{E}_{h}} h_{e}^{-1}\|[[\boldsymbol{v}_{h}]]\|_{0,e}^{2}.
 \end{eqnarray}
 \begin{remark}
 	Here we use $[[ \boldsymbol{v}_h ]]_{e}$ instead of $[[\boldsymbol{v}-\boldsymbol{v}_{h}]]_{e}$, since $[[\boldsymbol{v}]]_{e}=0$ for $\boldsymbol{v} \in$ $\boldsymbol{U}$.
 \end{remark}
 
 Next, we focus on proving the reliability of the error indicator defined in \eqref{eta}.
 The key of our argument is to use the space decomposition technique:
 decompose the DG finite element solution $\boldsymbol{u}_h$ into two parts: one is conforming part $\boldsymbol{u}_h^{conf}\in\boldsymbol{U}_h^{conf}:=\boldsymbol{U}_h\cap\boldsymbol{U}$ and the other is its $L^2$ orthogonal part $\boldsymbol{u}_h^{\bot}\in\boldsymbol{U}_h^{\bot}$.
 Therefore, we need to take care of continuous error $\|(\boldsymbol{u},p)-(\boldsymbol{u}_h^{conf},p_h)\|_{DG}$ instead of $\|(\boldsymbol{u},p)-(\boldsymbol{u}_h,p_h)\|_{DG}$.
 The main analysis tools for continuous error are Helmholtz decomposition and the two interpolations.
 We prove the reliability of the error indicator.
 
 The following lemmas provide some estimates related to the continuous error.
 \begin{lemma}\label{Lem:leqell_1ell_2}
 	Let $( \boldsymbol{u},{p}) \in  \boldsymbol{U}\times\mathbb{Q}$ be solution of system \eqref{Eqn:weak_1}-\eqref{Eqn:weak_2}, then for any  $(\boldsymbol{v}_h^{conf}, {p}_{h})\in \boldsymbol{U}^{conf}_h \times \mathbb{Q}_h$, we have
 	\begin{equation}\label{Eqn:errconf}
 	\|(\boldsymbol{u}-\boldsymbol{v}^{conf}_h, {p}-{p}_{h})\|_{\boldsymbol{U} \times \mathbb{Q}} \lesssim\|\tilde{\ell}_{1}\|_{\mathbb{Q}^{*}}+\|\tilde{\ell}_{2}\|_{\boldsymbol{U}^{*}},
 	\end{equation}
 	where
 	\begin{eqnarray}
 	&&\tilde{\ell}_{1}({q})=-a\left({p}_{h}, {q}\right)+b(\boldsymbol{v}^{conf}_h, {q}),  \forall q \in \mathbb{Q},\label{Eqn:ell_1} \\
 	&&\tilde{\ell}_{2}(\boldsymbol{v})=\ell_{2}(\boldsymbol{v})-d\left(\boldsymbol{v}, {p}_{h}\right)-c(\boldsymbol{v}^{conf}_h, \boldsymbol{v}), \forall \boldsymbol{v} \in \boldsymbol{U}. \label{Eqn:ell_2}
 	\end{eqnarray}
 \end{lemma}
 
 \begin{proof}
 	For the operator $\mathcal{A}$ given by \eqref{Eqn:Al}, it is easy to obtain the linearity, namely, for any $q_1,~q_2,~q\in\mathbb{Q}$ and  $\boldsymbol{v}_1,~\boldsymbol{v}_2,~\boldsymbol{v}\in\boldsymbol{U}$, we have
 	\begin{eqnarray*}
 		\left(\mathcal{A} \left({q}_{1}+{q}_{2},\boldsymbol{v}_{1}+\boldsymbol{v}_{2} \right)\right)({q},\boldsymbol{v})
 		=\left(\mathcal{A}\left({q}_{1},\boldsymbol{v}_{1} \right)\right)({q},\boldsymbol{v} )+\left(\mathcal{A}\left({q}_{2},\boldsymbol{v}_{2} \right)\right)({q},\boldsymbol{v}).
 	\end{eqnarray*}
 	Hence, we have
 	
 	\begin{eqnarray*}
 		\lefteqn{(\mathcal{A}({p}-{p}_{h},\boldsymbol{u}-\boldsymbol{v}^{conf}_h ))({q},\boldsymbol{v})} \\
 		&=&
 		(\mathcal{A}({p},\boldsymbol{u} ))({q},\boldsymbol{v} )
 		-
 		(\mathcal{A}({p}_{h},\boldsymbol{v}^{conf}_h ))({q},\boldsymbol{v}) \\
 		&=&
 		\ell_{2}(\boldsymbol{v})-(a({p}_{h}, {q})-b(\boldsymbol{v}^{conf}_h, {q})
 		+
 		d(\boldsymbol{v}, {p}_{h})
 		+
 		c(\boldsymbol{v}^{conf}_h,\boldsymbol{v})) \\
 		&:=&\tilde{\ell}_{1}({q})
 		+
 		\tilde{\ell}_{2}(\boldsymbol{v}).
 	\end{eqnarray*}
 	At last,
 	noting that $(\boldsymbol{u}-\boldsymbol{v}^{conf}_h,{p}-{p}_{h} ) \in \boldsymbol{U}\times\mathbb{Q}$ and using the definition of the operator norm, this completes the proof.
 \end{proof}
 
 In the following lemmas, our purpose is to obtian upper bounds for $\|\tilde{\ell}_{1}\|_{Q^{*}}$ and $\|\tilde{\ell}_{2}\|_{U^{*}}$ in Lemmas \ref{Lem:tildeell_1} and \ref{Eqn:tildeell_2}, respectively.
 \begin{lemma}\label{Lem:tildeell_1}
 	Let $\left(\boldsymbol{u}_{h},{p}_{h} \right) \in \boldsymbol{U}_{h}\times\mathbb{Q}_{h}$ be solution of  \eqref{Eqn:dweak_1}-\eqref{Eqn:dweak_2}. For any $\boldsymbol{v}_h^{conf}\in\boldsymbol{U}^{conf}_h$, we have
 	$$
 	\|\tilde{\ell}_{1}\|_{\mathbb{Q}^{*}}
 	\lesssim
 	\left(\sum_{\tau \in \mathcal{T}_h}\|R_{1}\left(\boldsymbol{u}_{h}, {p}_{h}\right)\|_{0,\tau}^{2}\right)^{1 / 2}
 	+
 	\left(\sum_{\tau \in \mathcal{T}_h}\|\mathrm{curl}_h(\boldsymbol{v}_h^{conf}-\boldsymbol{u}_{h})\|_{0,\tau}^{2}\right)^{1 / 2}.
 	$$
 \end{lemma}
 \begin{proof}
 	For any $q\in \mathbb{Q}$, by using  \eqref{Eqn:ell_1}, \eqref{Eqn:a} and \eqref{Eqn:b}, we have
 	\begin{eqnarray*}
 		\tilde{\ell}_{1}(q)
 		&=&
 		-a\left({p}_{h}, {q}\right)+b(\boldsymbol{v}^{conf}_h, {q})
 		\\
 		&=&
 		-\left(p_{h}, q\right)+(\mathrm{curl}~\boldsymbol{v}_h^{conf}, q)
 		\\
 		&=& \left({\mathrm{curl}_h~ \boldsymbol{u}}_{h}-p_{h}, q\right)_{\mathcal{T}_h}
 		+
 		(\mathrm{curl}_h(\boldsymbol{v}_h^{conf}-\boldsymbol{u}_{h}), q)_{\mathcal{T}_h}.
 	\end{eqnarray*}
 	Applying H\"older inequality and  Cauchy-Schwarz inequality leads to
 	\begin{eqnarray*}
 		\lefteqn{	
 			|\tilde{\ell}_{1}(q)|
 			\leq \sum_{\tau \in \mathcal{T}_h}\left\|\mathrm{curl}_h~ \boldsymbol{u}_{h}-{p}_{h}\right\|_{0,\tau}\|{q}\|_{0,\tau}
 			+
 			\sum_{\tau \in \mathcal{T}_h}\|\mathrm{curl}_h(\boldsymbol{v}^{conf}_{h}-\boldsymbol{u}_{h})\|_{0,\tau}\|{q}\|_{0,\tau}
 		}
 		\\
 		&&\leq 2 \left(\left(\sum_{\tau \in \mathcal{T}_h}\|R_{1}\left( \boldsymbol{u}_{h},p_{h}\right)\|_{0,\tau}^{2}\right)^{1 / 2}
 		+
 		\left(\sum_{\tau \in \mathcal{T}_h}\|\mathrm{curl}_h(\boldsymbol{u}_{h}-\boldsymbol{v}_h^{conf})\|_{0,\tau}^{2}\right)^{1 / 2}\right)\|q\|_{0,\Omega}.
 	\end{eqnarray*}
 	The proof is completed.
 \end{proof}
 
 In order to estimate the term $\|\tilde{\ell}_{2}\|_{\boldsymbol{U}^{*}}$ in Lemma \ref{Eqn:tildeell_2},
 we shall use the following two interpolation operators with the corresponding approximations.
 \begin{itemize}
 	\item[(1)] Scott-Zhang quasi-interpolation $S_h:H_0^1(\Omega)\rightarrow \{v\in C(\Omega)|\ v|_\tau \in{P}_1(\tau),~v|_{\partial\Omega}=0,~\forall \tau\in\mathcal{T}_h\}$, where ${P}_1(\tau)$ represents a linear polynomial space.
 	The definition and approximation properties of Scott-Zhang quasi-interpolation interpolation were first proposed in
 	\cite{ScottZhang90:483}.
 	For
 	$\psi\in H^1_0(\Omega)$, there hold
 	\begin{eqnarray}
 	\|\nabla S_h\psi\|_{0,\tau}
 	\lesssim
 	\|\nabla\psi\|_{0,\omega_\tau}, && \forall \tau\in\mathcal{T}_h,  \label{Ppsi} \\
 	\|\psi-S_h\psi\|_{0,\tau}
 	\lesssim
 	h_\tau \|\nabla \psi\|_{0,\omega_\tau}, && \forall \tau\in\mathcal{T}_h,  \label{psitau}\\
 	\|\psi-S_h\psi\|_{0,e}
 	\lesssim
 	h_e^{\frac{1}{2}} \|\nabla \psi\|_{0,\omega_e},&& \forall e\in\mathcal{E}_h,  \label{psie}
 	\end{eqnarray}
 	where $\omega_{\tau}:=\bigcup\limits_{\tau^{\prime} \cap \tau \neq \emptyset} \tau^{\prime}$ and  $\omega_{e}:=\bigcup\limits_{\tau \cap e \neq \emptyset} \tau$.
 	\item[(2)] Vector-Valued operator $
 	\boldsymbol{P}_{h}: \boldsymbol{H}^{1}(\Omega) \cap \boldsymbol{U} \rightarrow  \boldsymbol{U}_h^{conf}
 	$. The definition and approximation properties of Vector-Valued operator were proposed in
 	\cite{CarstensenHoppe05:19}.
 	For $q\in \boldsymbol{H}^1(\Omega)\cap\boldsymbol{U}$, there hold
 	\begin{eqnarray}
 	\|\boldsymbol{P}_h\boldsymbol{q}\|_{0,\tau} \lesssim \|\boldsymbol{q}\|_{1,\tilde{\omega}_\tau}, && \forall \tau\in\mathcal{T}_h,  \label{Pq}\\
 	\|\boldsymbol{q}-\boldsymbol{P}_h\boldsymbol{q}\|_{0,\tau} \lesssim h_\tau \|\boldsymbol{q}\|_{1,\tilde{\omega}_\tau}, && \forall \tau\in\mathcal{T}_h,  \label{qtau} \\
 	\|\boldsymbol{q}-\boldsymbol{P}_h\boldsymbol{q}\|_{0,e} \lesssim h_e^{\frac{1}{2}} \|\boldsymbol{q}\|_{1,\tilde{\omega}_e}, && \forall e\in\mathcal{E}_h, \label{qe}
 	\end{eqnarray}
 	where $\tilde{\omega}_{e}:=\bigcup\left\{\tau \in \mathcal{T}_{h}(\Omega) \mid e \in \mathcal{E}_{h}(T)\right\}$ and
 	$\tilde{\omega}_{\tau}:=\bigcup\left\{\omega_{e} \mid e \in \mathcal{E}_{h}(T)\right\}$.
 \end{itemize}

 \begin{lemma}\label{Eqn:tildeell_2}
 	Let $\left(\boldsymbol{u}_{h},{p}_{h}\right) \in \boldsymbol{U}_{h}\times\mathbb{Q}_{h} $  be solution of system \eqref{Eqn:dweak_1}-\eqref{Eqn:dweak_2}. For any $\boldsymbol{v}^{conf}_h\in\boldsymbol{U}^{conf}_h$, then there exists a constant $C_{1}>0$ depending only on $\|\beta\|_{0,\infty}$, we have
 	\begin{eqnarray*}
 		\|\tilde{\ell}_{2}\|_{\boldsymbol{U}^{*}}
 		&\leq&
 		C_{1} \left(\sum_{\tau \in \mathcal{T}_h} h_{\tau}^{2}\left(\|R_{2}\left(\boldsymbol{u}_{h}, {p}_{h}\right)\|_{0,\tau}^{2}
 		+
 		\|R_{3}\left(\boldsymbol{u}_{h}\right)\|_{0,\tau}^{2}\right)\right.\\
 		&&\left.+\sum_{e \in \mathcal{E}_h} h_{e}\left(\|J_{1}\left({p}_{h}\right)\|_{0,e}^{2}
 		+
 		\|J_{2}\left(\boldsymbol{u}_{h}\right)\|_{0,e}^{2}\right)
 		+
 		\sum_{\tau \in \mathcal{T}_h}\|\boldsymbol{u}_{h}-\boldsymbol{v}_h^{conf}\|_{0,\tau}^{2}\right)^{1 / 2}.
 	\end{eqnarray*}
 \end{lemma}
 
 \begin{proof}
 	For any $\boldsymbol{v}\in\boldsymbol{U}$, we use Helmholtz decomposition $\boldsymbol{v}$ as follows (see the Theorem 2.1 of 	\cite{CarstensenHoppe05:19})
 	\begin{equation}\label{wdec}
 	\boldsymbol{v}:= \boldsymbol{v}^0 + \boldsymbol{v}^{\bot},
 	\end{equation}
 	and
 	\begin{equation}
 	\|\boldsymbol{v}^{0}\|^2_{0,\Omega} + \|\boldsymbol{v}^{\bot}\|^2_{0,\Omega} = \|\boldsymbol{v}\|^2_{0,\Omega},
 	\end{equation}
 	where $\boldsymbol{v}^0 \in \boldsymbol{H}_0(\mathrm{curl}0,\Omega): =\{\boldsymbol{v}:\boldsymbol{v}\in \boldsymbol{H}_0(\mathrm{curl}, \Omega), \mathrm{curl}\ \boldsymbol{v}=0 \}$ and $\boldsymbol{v}^{\bot}\in \boldsymbol{H}^{\bot}(\mathrm{curl},\Omega): =\{\boldsymbol{v}:\boldsymbol{v}\in \boldsymbol{H}_0(\mathrm{curl}, \Omega), (\boldsymbol{v},\boldsymbol{v}^0)=0,\ \boldsymbol{v}^0\in \boldsymbol{H}_0(\mathrm{curl}0,\Omega)\}$. We make use of the representation $\boldsymbol{H}_0(\mathrm{curl}0,\Omega) = \boldsymbol{grad}\ {H}^1_0(\Omega)$, hence, we have $\boldsymbol{v}^0=\nabla \psi$ for some $\psi\in H_0^1(\Omega)$.
 	%Note $\left(\nabla\varphi+\boldsymbol{z}\right) \in  \boldsymbol{H} _0(\mathrm{curl}; \Omega)$, then for $e\in \mathcal{E}_h$, there have $[[\nabla\varphi+\boldsymbol{z}]]_e=0$.
 	Applying the definition of $\tilde{\ell}_2$ \eqref{Eqn:ell_2}, we have
 	\begin{eqnarray}\nonumber
 	\tilde{\ell}_2(\boldsymbol{v})
 	&=&
 	\tilde{\ell}_2(\boldsymbol{v}^0+\boldsymbol{v}^{\bot})
 	=
 	\tilde{\ell}_2(\nabla\psi+\boldsymbol{v}^{\bot})
 	\\ \label{Eqn:ell_dec}
 	&=&
 	\tilde{\ell}_2(\nabla S_h\psi+\boldsymbol{P}_h\boldsymbol{v}^{\bot}) +
 	\tilde{\ell}_2(\nabla\psi+\boldsymbol{v}^{\bot}-\nabla S_h\psi-\boldsymbol{P}_h\boldsymbol{v}^{\bot}).
 	\end{eqnarray}
 	
 	We first analyze the first item on the right hand side of  \eqref{Eqn:ell_dec}. According to the definition of $\tilde{\ell}_2$ \eqref{Eqn:ell_2}, we have
 	\begin{eqnarray}\label{Eqn:S}
 	\lefteqn{\tilde{\ell}_{2}\left(\nabla S_h\psi+\boldsymbol{P}_h\boldsymbol{v}^{\bot}\right) } \nonumber\\
 	&=&
 	\ell_{2}\left(\nabla S_h\psi+\boldsymbol{P}_h\boldsymbol{v}^{\bot}\right)
 	-
 	d\left(\nabla S_h\psi+\boldsymbol{P}_h\boldsymbol{v}^{\bot}, {p}_{h}\right)
 	-
 	c(\boldsymbol{v}^{conf}_h, \nabla S_h\psi+\boldsymbol{P}_h\boldsymbol{v}^{\bot})
 	\nonumber\\
 	&=&
 	\ell_{2}\left(\nabla S_h\psi+\boldsymbol{P}_h\boldsymbol{v}^{\bot}\right)
 	-
 	d\left(\nabla S_h\psi+\boldsymbol{P}_h\boldsymbol{v}^{\bot}, {p}_{h}\right)
 	-
 	c\left(\boldsymbol{u}_h,\nabla S_h\psi+\boldsymbol{P}_h\boldsymbol{v}^{\bot} \right) \nonumber \\
 	&&
 	+
 	c(\boldsymbol{u}_{h}-\boldsymbol{v}^{conf}_h, \nabla S_h\psi+\boldsymbol{P}_h\boldsymbol{v}^{\bot}).
 	\end{eqnarray}
 	By using Remark \ref{Rem:1}, for any $e\in\mathcal{E}_h$,
 	noting that $[[\nabla S_h\psi+\boldsymbol{P}_h\boldsymbol{v}^{\bot}]]_e=0$ since $\nabla S_h\psi+\boldsymbol{P}_h\boldsymbol{v}^{\bot}\subset \boldsymbol{U}$, and choosing $\boldsymbol{v}_h=\boldsymbol{P}_h\boldsymbol{v}^{\bot}+\nabla S_h\psi$ in \eqref{Eqn:dweak_2}, we have
 	\begin{equation}\label{Eqn:0}
 	\ell_{2}\left(\boldsymbol{P}_h\boldsymbol{v}^{\bot}+\nabla S_h\psi\right)-d\left(\boldsymbol{P}_h\boldsymbol{v}^{\bot}+\nabla S_h\psi, {p}_{h}\right)-c\left(\boldsymbol{u}_{h},  \boldsymbol{P}_h\boldsymbol{v}^{\bot}+\nabla S_h\psi\right)=0.
 	\end{equation}
 	Combining \eqref{Eqn:S} and \eqref{Eqn:0}, we have
 	\begin{eqnarray*}
 		\tilde{\ell}_{2}\left(\nabla S_h\psi+\boldsymbol{P}_h\boldsymbol{v}^{\bot}\right)
 		&=& c(\boldsymbol{u}_{h}-\boldsymbol{v}^{conf}_h, \nabla S_h\psi + \boldsymbol{P}_h\boldsymbol{v}^{\bot}).
 	\end{eqnarray*}
 	Applying Cauchy-Schwarz inequality, \eqref{Ppsi}, \eqref{Pq} and $\boldsymbol{H}_0^{\bot}(\mathrm{curl},\Omega)$ is continuously imbedded in $\boldsymbol{H}^1(\Omega)\cap\boldsymbol{H}_0(\mathrm{curl},\Omega)$(see \cite{CarstensenHoppe05:19}),
 	we deduce
 	\begin{eqnarray}\label{Eqn:ell_2Pi}
 	\tilde{\ell}_{2}\left(\nabla S_h\psi+\boldsymbol{P}_h\boldsymbol{v}^{\bot}\right)
 	&\lesssim&
 	C_1\|\boldsymbol{u}_{h}-\boldsymbol{v}^{conf}_h\|_{0,\Omega}
 	(\|\nabla S_h\psi\|_{0,\Omega}
 	+
 	\|\boldsymbol{P}_h\boldsymbol{v}^{\bot}\|_{0,\Omega}) \nonumber\\
 	&\lesssim&
 	C_1\|\boldsymbol{u}_{h}-\boldsymbol{v}^{conf}_h\|_{0,\Omega}
 	(\|\nabla\psi\|_{0,\Omega} + \|  \boldsymbol{v}^{\bot}\|_{0,\Omega}) \nonumber\\
 	&\lesssim&
 	C_1\|\boldsymbol{u}_{h}-\boldsymbol{v}^{conf}_h\|_{0,\Omega}
 	(\|\nabla\psi\|_{0,\Omega} + \| \mathrm{curl}~ \boldsymbol{v}^{\bot}\|_{0,\Omega}) \nonumber\\
 	&\lesssim&
 	C_1\|\boldsymbol{u}_{h}-\boldsymbol{v}^{conf}_h\|_{0,\Omega}\|\boldsymbol{v}\|_{\mathrm{curl},\Omega},
 	\end{eqnarray}
 	where $C_{1}>0$ depends only on $\|\beta\|_{0,\infty}$.
 	
 	Next, we analyze the second item on the right hand side of  \eqref{Eqn:ell_dec}, using the definition of $\tilde{\ell}_2$ \eqref{Eqn:ell_2}, the fact $(\mathrm{curl}_h(\nabla\psi - \nabla S_h\psi),\alpha p_h)_{\mathcal{T}_h}=0$ and Green formula, we have
 	\begin{eqnarray}\label{Eqn:ell_2dec}
 	\lefteqn{\tilde{\ell}_{2}\left(
 		\nabla\psi+\boldsymbol{v}^{\bot}-\nabla S_h\psi-\boldsymbol{P}_h\boldsymbol{v}^{\bot}\right)} \nonumber\\
 	&=&
 	\tilde{\ell}_{2}\left(
 	\nabla\psi-\nabla S_h\psi\right)
 	+
 	\tilde{\ell}_{2}\left(
 	\boldsymbol{v}^{\bot}-\boldsymbol{P}_h\boldsymbol{v}^{\bot}\right)    \nonumber\\ 	
 	&=&
 	(\boldsymbol{f},\nabla\psi - \nabla S_h\psi)_{\mathcal{T}_h}
 	- (\mathrm{curl}_h(\nabla\psi - \nabla S_h\psi),\alpha p_h)_{\mathcal{T}_h} - (\beta \boldsymbol{v}^{conf}_h,\nabla\psi -                \nabla S_h\psi)_{\mathcal{T}_h} \nonumber\\
 	&& + (\boldsymbol{f},\boldsymbol{v}^{\bot} -\boldsymbol{P}_h\boldsymbol{v}^{\bot})_{\mathcal{T}_h}
 	- (\mathrm{curl}_h(\boldsymbol{v}^{\bot} -\boldsymbol{P}_h\boldsymbol{v}^{\bot}),\alpha p_h)_{\mathcal{T}_h}
 	- (\beta\boldsymbol{v}^{conf}_h,\boldsymbol{v}^{\bot} - \boldsymbol{P}_h\boldsymbol{v}^{\bot})_{\mathcal{T}_h} \nonumber\\
 	&=& (\boldsymbol{f}, \nabla\psi - \nabla S_h\psi)_{\mathcal{T}_h}
 	- (\beta \boldsymbol{u}_h,\nabla\psi - \nabla S_h\psi)_{\mathcal{T}_h}
 	+ (\beta(\boldsymbol{u}_h-\boldsymbol{v}^{conf}_h), \nabla\psi - \nabla S_h\psi)_{\mathcal{T}_h} \nonumber \\
 	&&+(\boldsymbol{f},\boldsymbol{v}^{\bot} - \boldsymbol{P}_h\boldsymbol{v}^{\bot})_{\mathcal{T}_h} - (\mathrm{curl}_h(\boldsymbol{v}^{\bot}-\boldsymbol{P}_h\boldsymbol{v}^{\bot}), \alpha p_h)_{\mathcal{T}_h}
 	- (\beta \boldsymbol{u}_h,\boldsymbol{v}^{\bot} - \boldsymbol{P}_h\boldsymbol{v}^{\bot})_{\mathcal{T}_h} \nonumber\\
 	&&+ (\beta (\boldsymbol{u}_h-\boldsymbol{v}^{conf}_h),\boldsymbol{v}^{\bot} - \boldsymbol{P}_h\boldsymbol{v}^{\bot})_{\mathcal{T}_h} \nonumber \\
 	&=& -\left(\nabla \cdot \left(\boldsymbol{f}-\beta \boldsymbol{u}_h\right),  \psi- S_h\psi \right)_{\mathcal{T}_h}
 	+
 	\left\langle [\boldsymbol{f}-\beta\boldsymbol{u}_h ]_e, \psi- S_h\psi \right\rangle_{\mathcal{E}_h}
 	\nonumber\\
 	&&+
 	(\beta (\boldsymbol{u}_h-\boldsymbol{v}^{conf}_h), \nabla\psi - \nabla S_h\psi)_{\mathcal{T}_h}
 	+
 	\left(\boldsymbol{f}-\mathbf{curl}_h~\alpha p_h-\beta\boldsymbol{u}_h, \boldsymbol{v}^{\bot} - \boldsymbol{P}_h \boldsymbol{v}^{\bot}   \right)_{\mathcal{T}_h}
 	\nonumber \\
 	&&	-
 	\left\langle [[\alpha p_h]]_e, \boldsymbol{v}^{\bot} - \boldsymbol{P}_h \boldsymbol{v}^{\bot}  \right\rangle_{\mathcal{E}_h}
 	+
 	(\beta(\boldsymbol{u}_h-\boldsymbol{v}^{conf}_h), \boldsymbol{v}^{\bot} - \boldsymbol{P}_h \boldsymbol{v}^{\bot} )_{\mathcal{T}_h}.
 	\end{eqnarray}
 	
 	Applying Cauchy-Schwarz inequality, \eqref{psitau}-\eqref{qe} and $\boldsymbol{H}_0^{\bot}(\mathrm{curl},\Omega)$ is continuously imbedded in $\boldsymbol{H}^1(\Omega)\cap\boldsymbol{H}_0(\mathrm{curl},\Omega)$, we deduce
 	\begin{eqnarray*}
 		\lefteqn{ |\tilde{\ell}_{2}(\boldsymbol{v})|
 			\lesssim
 			C_1 \left(\sum_{\tau \in \mathcal{T}_h} h_{\tau}^{2}\left(\|R_{2}\left(\boldsymbol{u}_{h}, {p}_{h}\right)\|_{0,\tau}^{2}
 			+
 			\left\|R_{3}(\boldsymbol{u}_{h}\right)\|_{0,\tau}^{2}\right)\right.
 		} \\
 		&&
 		\left. \ \  \  +\sum_{e \in \mathcal{E}_h} h_{e}\left(\|J_{1}\left({p}_{h}\right)\|_{0,e}^{2}
 		+
 		\|J_{2}\left(\boldsymbol{u}_{h}\right)\|_{0,e}^{2}\right)
 		+
 		\sum_{\tau \in \mathcal{T}}\|\boldsymbol{u}_{h}-\boldsymbol{v}^{conf}_h\|_{0,\tau}^{2}\right)^{1 / 2}\|\boldsymbol{v}\|_{\mathrm{curl},\Omega},
 	\end{eqnarray*}
 	which complete the proof.
 \end{proof}
 
 In the end of this section, we give the upper bound of the error estimator.
 \begin{theorem}\label{The:Upper}
 	Let $(\boldsymbol{u},p)\in\boldsymbol{U}\times\mathbb{Q}$ and $(\boldsymbol{u}_h,p_h)\in\boldsymbol{U}_h\times\mathbb{Q}_h$ be the solutions of  \eqref{Eqn:weak_1}-\eqref{Eqn:weak_2} and \eqref{Eqn:dweak_1}-\eqref{Eqn:dweak_2}, respectively. Then there exists a constant $C_{1}>0$ depending only on $\|\beta\|_{\infty,\Omega}$ and the grid shape regularity, we have
 	\begin{equation*}
 	\| (\boldsymbol{u},p) - (\boldsymbol{u}_h,p_h) \|^2_{DG}
 	\lesssim  C_{1}\eta^2(\boldsymbol{u}_h,p_h;\mathcal{T}_h).
 	\end{equation*}
 \end{theorem}
 \begin{proof}
 	For any $\boldsymbol{v}^{conf}_h\in \boldsymbol{U}_h^{conf}$, using the definition of $\|\cdot\|_{DG}$ in \eqref{Eqn:DG}, triangle inequality,  we get
 	\begin{eqnarray}
 	\lefteqn{\|(\boldsymbol{u}, {p})-\left(\boldsymbol{u}_{h}, {p}_{h}\right)\|_{D G}^{2}} \nonumber \\
 	&&=\|{p}-{p}_{h}\|_{0,\Omega}^{2}
 	+\|\boldsymbol{u}-\boldsymbol{u}_{h}\|_{0,\Omega}^{2}
 	+\|\mathrm{curl}~\boldsymbol{u}-\mathrm{curl}_h~\boldsymbol{u}_{h}\|_{0,\Omega}^{2}
 	+
 	\kappa\sum\limits_{e \in \mathcal{E}_{h}} h_{e}^{-1}\|[[\boldsymbol{u}_{h}]]\|_{0,e}^{2} \nonumber\\
 	&&=|{p}-{p}_{h}\|_{0,\Omega}^{2}
 	+\|\boldsymbol{u}-\boldsymbol{u}_{h}\|_{\mathrm{curl},\Omega}^{2}
 	+\kappa\sum\limits_{e \in \mathcal{E}_{h}} h_{e}^{-1}\|[[\boldsymbol{u}_{h}]]\|_{0,e}^{2} \nonumber\\
 	&&\leq \|p-{p}_{h}\|_{0,\Omega}^{2}+\|\boldsymbol{u}-\boldsymbol{v}_h^{conf}\|_{\mathrm{curl},\Omega}^{2}
 	+
 	\|\boldsymbol{v}^{conf}_h-\boldsymbol{u}_{h}\|_{\mathrm{curl},\Omega}^{2}
 	+\kappa\sum\limits_{e \in \mathcal{E}_{h}} h_{e}^{-1}\|[[\boldsymbol{u}_{h}]]\|_{0,e}^{2} \nonumber\\
 	&&=	\|(\boldsymbol{u}-\boldsymbol{v}^{conf}_h, {p}-{p}_{h})\|_{\boldsymbol{U} \times \mathbb{Q}}
 	+
 	\|\boldsymbol{v}^{conf}_h-\boldsymbol{u}_{h}\|_{\mathrm{curl},\Omega}^{2}
 	+\kappa\sum\limits_{e \in \mathcal{E}_{h}} h_{e}^{-1}\|[[\boldsymbol{u}_{h}]]\|_{0,e}^{2}.\label{Eqn:up-uhph}
 	\end{eqnarray}
 	
 	For $\boldsymbol{u}_h\in \boldsymbol{U}_h$, there exists $\boldsymbol{u}^{conf}_{h} \in \boldsymbol{U}_h^{conf}$, $\boldsymbol{u}^{\bot}_{h} \in \boldsymbol{U}_h^{\bot}$ satisfing (see Proposition 4.10 of
 	\cite{LohrengelNicaise07:27})
 	\begin{equation*}\label{u_h}
 	\boldsymbol{u}_h=\boldsymbol{u}^{conf}_{h}+\boldsymbol{u}^{\bot}_{h},
 	\end{equation*}
 	and
 	%	\begin{eqnarray}\label{Eqn:v_c}
 	%		h^2_\tau\left\|\mathrm{curl}_h\ \boldsymbol{u}^{\bot}_{h}\right\|^2_{0,\Omega}
 	%		+
 	%		\left\|\boldsymbol{u}^{\bot}_{h}\right\| ^2_{0,\Omega}
 	%		\lesssim
 	%		h^2_\tau\sum\limits_{e\in \mathcal{E}_h}h^{-1}_e \left\| [[ \boldsymbol{u}_{h}]] _e\right\|^2_{0, e}.
 	%	\end{eqnarray}
 	%	Combining \eqref{u_h} and \eqref{Eqn:v_c}, we have
 	\begin{eqnarray}\label{Eqn:v_c2}
 	h^2_\tau\|\mathrm{curl}_h(\boldsymbol{u}_{h}-\boldsymbol{u}^{conf}_{h})\|^2_{0,\Omega}
 	+
 	\|\boldsymbol{u}_h-\boldsymbol{u}^{conf}_{h}\| ^2_{0,\Omega}
 	\lesssim
 	h^2_\tau\sum\limits_{e\in \mathcal{E}_h}h^{-1}_e \left\| [[ \boldsymbol{u}_{h}]] _e\right\|^2_{0, e}.
 	\end{eqnarray}
 	
 	At last, choosing $\boldsymbol{v}_h^{conf} = \boldsymbol{u}_h^{conf}$ in \eqref{Eqn:up-uhph} and using Lemmas \ref{Lem:leqell_1ell_2}, \ref{Lem:tildeell_1}, \ref{Eqn:tildeell_2} and \eqref{Eqn:v_c2}, we have
 	\begin{eqnarray*}
 		\lefteqn{\|(\boldsymbol{u}, {p})-\left(\boldsymbol{u}_{h}, {p}_{h}\right)\|_{D G}^{2}} \\
 		&\leq&\|(\boldsymbol{u}-\boldsymbol{u}^{conf}_h, {p}-{p}_{h})\|_{\boldsymbol{U} \times \mathbb{Q}}
 		+
 		\|\boldsymbol{u}^{conf}_h-\boldsymbol{u}_{h}\|_{\mathrm{curl},\Omega}^{2}
 		+\kappa\sum\limits_{e \in \mathcal{E}_{h}} h_{e}^{-1}\|[[\boldsymbol{u}_{h}]]\|_{0,e}^{2}\\
 		&\lesssim&
 		\|\tilde{\ell}_{1}\|_{\mathbb{Q}^{*}}^{2}+\|\tilde{\ell}_{2}\|_{\boldsymbol{U}^{*}}^{2}
 		+
 		\|\boldsymbol{u}^{conf}_h-\boldsymbol{u}_{h}\|_{\mathrm{curl},\Omega}^{2}
 		+
 		\kappa\sum\limits_{e \in \mathcal{E}_{h}} h_{e}^{-1}\|[[\boldsymbol{u}_{h}]]\|_{0,e}^{2}\\
 		&\lesssim& C_{1}\eta^2(\boldsymbol{v}_h,p_h;\mathcal{T}_h),
 	\end{eqnarray*}
 	which  completes the proof.
 \end{proof}

 %%%%%%%
 %%%%%%%
 
 \section{Efficiency analysis}\label{sec:5}
 
 In this section, we will prove that the error estimator is efficient. In the Section \ref{sec:4}, we have already proved that the error estimator is reliable. Hence, the error estimator is a good indicator of energy error, that is
 \begin{equation}\label{Eqn:goodind}
 \eta^{2}\left(\boldsymbol{u}_{h},p_{h} ; \mathcal{T}_h\right)
 \approx
 \|\left(p-p_{h}, \boldsymbol{u}-\boldsymbol{u}_{h}\right)\|_{\mathrm{DG}}^{2}.
 \end{equation}
 
 We fitst introduce the following theorem, which is the efficiency of the error estimator.
 \begin{theorem}\label{The:eff}
 	Let $(\boldsymbol{u},p)\in\boldsymbol{U}\times\mathbb{Q}$ and $(\boldsymbol{u}_h,p_h)\in\boldsymbol{U}_h\times\mathbb{Q}_h$ be the solutions of  \eqref{Eqn:weak_1}-\eqref{Eqn:weak_2} and \eqref{Eqn:dweak_1}-\eqref{Eqn:dweak_2}, respectively.
 	The space of all piecewise polynomial of arbitrary order on $\mathcal{T}_h$ is denoted by $P(\mathcal{T}_h)$.
 	Then for any $\boldsymbol{f}_{h} \in \mathcal{S}_{h}^{\text {conf }}:=P(\mathcal{T}_h) \cap \boldsymbol{H}(\operatorname{div} ; \Omega)$, we have
 	$$
 	\eta^{2}\left(\boldsymbol{u}_{h} ,p_{h} ; \mathcal{T}_h\right)
 	\lesssim
 	C_2\Big(
 	\|\left(p-p_{h}, \boldsymbol{u}-\boldsymbol{u}_{h}\right)\|_{\mathrm{DG}}^{2}
 	+
 	\|\boldsymbol{f}-\boldsymbol{f}_{h}\|_{0,\Omega}^{2}
 	+
 	\sum_{\tau \in \mathcal{T}_h} h_{\tau}^{2}\|\nabla \cdot\left(\boldsymbol{f}-\boldsymbol{f}_{h}\right)\|_{0,\tau}^{2}
 	\Big),
 	$$
 	where $C_2$ is a constant depending on  $\|\alpha\|_{0,\infty}$ and  $\|\beta\|_{0,\infty}$.
 \end{theorem}
 \begin{remark}
 	We assume $\boldsymbol{f} \in \boldsymbol{H}(\operatorname{div} ; \Omega)$ to show the proof of the Theorem \ref{The:eff}.
 	While, we should note that the conditional $\boldsymbol{f}_{h} \in \mathcal{S}_{h}^{\text {conf }}$ is only required in the Lemma \ref{lem:J2} ( See \eqref{f-f_h}).
 \end{remark}
 For any $\tau \in \mathcal{T}_{h}$, $e \in \mathcal{E}_{h}$, let $b_{\tau}$ and $b_{e}$ be standard the interior bubble functions and edge bubble functions, respectively. We introduce the following two lemmas which are used in later proofs.
 
 \begin{lemma}[\cite{HoustonPerugia07:122}, Lemma 5.1]\label{Lem:3.4}
 	Let $\chi$ be a scalar polynomial function on $\tau$, then
 	\begin{eqnarray}
 	&
 	\|b_\tau \chi\|_{0,\tau}
 	\lesssim \|\chi\|_{0,\tau},\label{Eqn:3.23} \\
 	&
 	\| \chi\|_{0,\tau}  \lesssim
 	\|b_\tau^{\frac{1}{2}}\chi\|_{0,\tau},     \label{Eqn:3.24}\\
 	&
 	\|\nabla \left(b_{\tau}\chi\right) \|_{0,\tau}  \lesssim
 	h_\tau^{-1}\|\chi\|_{0,\tau}. \label{Eqn:3.25}
 	\end{eqnarray}
 	
 	Moreover, let $e$ be an edge shared by two triangles $\tau_1$ and $\tau_2$, let $\phi$  be a scalar polynomial function on $e$, the
 	\begin{equation}\label{Eqn:3.26}
 	\|\phi\|_{0,e}
 	\lesssim \|b_e^{\frac{1}{2}} \phi \|_{0,e} .
 	\end{equation}
 	
 	In addition, there exists an extension of $b_e\phi$ as $\tilde{\phi}_{b} \in H_{0}^{1}\left(\left(\overline{\tau}_1 \cup \overline{\tau}_2\right)^{\circ}\right)$ such that $ \tilde{\phi}_{b}|_e=b_e\phi$  and
 	\begin{eqnarray}\label{Eqn:3.27}
 	&
 	\|\tilde{\phi}_{b}\|_{0,\tau}
 	\lesssim h_e^{\frac{1}{2}}\|\phi\|_{0,e}, \quad\quad \forall \tau\in \omega_e, \\  \label{Eqn:3.28}
 	&
 	\|\nabla \tilde{\phi}_{b}\|_{0,\tau}
 	\lesssim h_e^{-\frac{1}{2}}\|\phi\|_{0,e},  \quad \forall \tau \in \omega_e.
 	\end{eqnarray}
 \end{lemma}
 
 \begin{lemma}[\cite{HoustonPerugia07:122}, Lemma 5.2]\label{Lem:3.5}
 	Let $\boldsymbol{w}$ be a vector-valued polynomial function on $\tau$, then
 	\begin{eqnarray}
 	&
 	\|b_\tau \boldsymbol{w}\|_{0,\tau} \lesssim\|\boldsymbol{w}\|_{0,\tau},\label{Eqn:3.29}
 	\\
 	&
 	\| \boldsymbol{w}\|_{0,\tau}\lesssim \|b_\tau^{\frac{1}{2}} \boldsymbol{w}\|_{0,\tau}, \label{Eqn:3.30}
 	\\
 	&
 	\|\mathrm{curl}\left(b_\tau \boldsymbol{w}\right)\|_{0,\tau} \lesssim h_\tau^{-1}\|\boldsymbol{w}\|_{0,\tau}.\label{Eqn:3.31}
 	\end{eqnarray}
 	
 	Moreover, let $e$ be an edge shared by two triangles on $\tau_1$ and $\tau_2$, and let $\boldsymbol{w}$ be a vector-valued polynomial function on $e$, then
 	\begin{equation}\label{Eqn:3.32}
 	\|\boldsymbol{w}\|_{0,e} \lesssim  \|b_e^{\frac{1}{2}} \boldsymbol{w} \|_{0,e}.
 	\end{equation}
 	
 	In addition, there exists an extension of $b_e\boldsymbol{w}$ as $\tilde{\boldsymbol{w}}_{b} \in \left(H_{0}^{1}\left(\left(\overline{\tau}_1 \cup \overline{\tau}_2\right)^{\circ}\right)\right)^{2}$, such that $\left.\tilde{\boldsymbol{w}}_{b}\right|_e=b_e\boldsymbol{w}$, and
 	\begin{eqnarray}\label{Eqn:3.33}
 	&
 	\|\tilde{\boldsymbol{w}}_{b}c\|_{0,\tau} \lesssim h_e^{\frac{1}{2}}\|\boldsymbol{w}\|_{0,e},  \quad \quad\forall \tau \in \omega_e,
 	\\  \label{Eqn:3.34}
 	&
 	\|\mathrm{curl}~\tilde{\boldsymbol{w}}_{b}\|_{0,\tau} \lesssim h_e^{-\frac{1}{2}}\|\boldsymbol{w}\|_{0,e},  \quad \forall \tau \in \omega_e.
 	\end{eqnarray}
 \end{lemma}
 
 Next, we estimate each term of the error indicator $\eta(u_h,p_h;\mathcal{T}_h)$ given by \eqref{eta}, separately. We first estimate the first term $R_{1}$ of the error estimator.
 
 \begin{lemma}\label{Lemm:R1}
 	Let $(\boldsymbol{u},p)\in\boldsymbol{U}\times\mathbb{Q}$ and $(\boldsymbol{u}_h,p_h)\in\boldsymbol{U}_h\times\mathbb{Q}_h$ be the solutions of  \eqref{Eqn:weak_1}-\eqref{Eqn:weak_2} and \eqref{Eqn:dweak_1}-\eqref{Eqn:dweak_2}, respectively. Then we obtain
 	\begin{equation}\label{eqn:R1}
 	\sum_{\tau \in \mathcal{T}_h}\|R_{1}\left(\boldsymbol{u}_{h},p_{h}  \right)\|_{0,\tau}^{2}
 	\leqslant
 	\|\left(p-p_{h}, \boldsymbol{u}-\boldsymbol{u}_{h}\right)\|_{\mathrm{DG}}^{2}.
 	\end{equation}
 \end{lemma}
 
 \begin{proof}
 	We can observe that $R_{1}\left(\boldsymbol{u}_{h},p_{h} \right)=p_{h}-\mathrm{curl}_h \boldsymbol{u}_{h} \in \mathbb{Q}_h\subset\mathbb{Q}$, and taking $q=R_{1}\left( \boldsymbol{u}_{h},p_{h} \right)$ in \eqref{Eqn:weak_1}, we have
 	\begin{equation}\label{Eqn:qh-curluh}
 	\left(p-\operatorname{curl} \boldsymbol{u}, R_{1}\left(\boldsymbol{u}_{h},p_{h} \right)\right)_{\mathcal{T}_h}=0.
 	\end{equation}
 	Thus, by \eqref{Eqn:qh-curluh} and Cauchy-Schwarz inequality, we obtain
 	\begin{eqnarray*}
 		\sum_{\tau \in \mathcal{T}_h}\|R_{1}\left(\boldsymbol{u}_{h},p_{h} \right)\|_{0,\tau}^{2}
 		&=&
 		\left(p_{h}-\mathrm{curl}_h~\boldsymbol{u}_{h}, R_{1}\left( \boldsymbol{u}_{h},p_{h}\right)\right)_{\mathcal{T}_h}\\
 		&=&
 		\left(p_{h}-p-\mathrm{curl}_h\left(\boldsymbol{u}_{h}-\boldsymbol{u}\right), R_{1}\left(\boldsymbol{u}_{h},p_{h} \right)\right)_{\mathcal{T}_h} \\
 		&\leqslant&
 		\left(\|p_{h}-p\|_{0,\Omega}+\|\mathrm{curl}_h\left(\boldsymbol{u}_{h}-\boldsymbol{u}\right)\|_{0,\Omega}\right)\|R_{1}\left( \boldsymbol{u}_{h},p_{h}\right)\|_{0,\Omega},
 	\end{eqnarray*}
 	which implies \eqref{eqn:R1}.
 \end{proof}
 
 We estimate the second term $R_2$ of the error estimator $\eta(u_h,p_h;\mathcal{T}_h)$.
 \begin{lemma}\label{Lemm:R2}
 	Let $(\boldsymbol{u},p)\in\boldsymbol{U}\times\mathbb{Q}$ and $(\boldsymbol{u}_h,p_h)\in\boldsymbol{U}_h\times\mathbb{Q}_h$ be the solutions of  \eqref{Eqn:weak_1}-\eqref{Eqn:weak_2} and \eqref{Eqn:dweak_1}-\eqref{Eqn:dweak_2}, respectively. Then we have
 	\begin{equation*}
 	\sum_{\tau \in \mathcal{T}_h} h_{\tau}^{2}\|R_{2}\left(\boldsymbol{u}_{h},p_{h} \right)\|_{0,\tau}^{2}
 	\lesssim
 	C_2\Big(
 	\left\| (\boldsymbol{u},p)-(\boldsymbol{u}_{h},p_h)\right\|_{\mathrm{DG}}^{2}
 	+
 	\sum_{\tau \in \mathcal{T}_h} h_{\tau}^{2}\left\|\boldsymbol{f}-\boldsymbol{f}_{h}\right\|_{0,\tau}^{2}
 	\Big),
 	\end{equation*}
 	where $C_2$ is a constant depending on  $\|\alpha\|_{0,\infty}$ and  $\|\beta\|_{0,\infty}$.
 \end{lemma}
 \begin{proof}
 	Let $\boldsymbol{f}_{h} \in \mathcal{S}_{h}^{\mathrm{conf }}$ and let $\boldsymbol{w}:=\boldsymbol{f}_{h}-\boldsymbol{curl}_h~ \alpha p_{h} - \beta\boldsymbol{u}_{h}$. Hence $\boldsymbol{w}$ is also a polynomial on $\tau$. Let $\boldsymbol{w}_{b}=b_{\tau} \boldsymbol{w}$, we can observe that $\boldsymbol{w}_{b} \in \boldsymbol{H}_{0}^{1}(\tau)  \subset \boldsymbol{H}_{0}^{1}(\Omega)  \subset \boldsymbol{U}$. Setting $\boldsymbol{v}=\boldsymbol{w}_{b}$ in \eqref{Eqn:weak_2} and using Green's formula, we get
 	\begin{equation}\label{Eqn:fh-curlaph-uh}
 	\left(\boldsymbol{f} -\boldsymbol{curl}~ \alpha p - \beta\boldsymbol{u}, \boldsymbol{w}_{b}\right)_\tau=0.
 	\end{equation}
 	Then by using \eqref{Eqn:3.30}, the definitons of $\boldsymbol{w}$ and $\boldsymbol{w}_b$, \eqref{Eqn:fh-curlaph-uh} and Cauchy-Schwarz inequality, we have
 	\begin{eqnarray}\label{R2_v}
 	\|\boldsymbol{w}\|_{0,\tau}^{2}
 	&\lesssim&
 	\|b_{\tau}^{1 / 2} \boldsymbol{w}\|_{0,\tau}^{2}
 	=
 	(\boldsymbol{w},b_\tau\boldsymbol{w})_\tau
 	\nonumber\\
 	&=&
 	(\left(\boldsymbol{f}_{h}-\boldsymbol{curl}_h~\alpha p_{h}
 	- \beta\boldsymbol{u}_{h}\right) , \boldsymbol{w}_{b})_\tau \nonumber\\
 	&=&
 	\left(\left(\boldsymbol{f}_{h}-\boldsymbol{f}\right)
 	-
 	\boldsymbol{curl}_h~\alpha\left(p_{h}-p\right)
 	-
 	\beta\left(\boldsymbol{u}_{h}-\boldsymbol{u}\right), \boldsymbol{w}_{b}\right)_\tau  \nonumber\\
 	&\lesssim&
 	\left(\left\|\boldsymbol{f}_{h}-\boldsymbol{f}\right\|_{0,\tau}
 	+
 	\|\beta\|_{0,\infty}\left\|\boldsymbol{u}_{h}-\boldsymbol{u}\right\|_{0,\tau}\right)\|\boldsymbol{w}_b\|_{0,\tau}
 	\nonumber\\
 	&& -
 	\left(\boldsymbol{curl}_h~\alpha\left(p_{h}-p\right)  , \boldsymbol{w}_{b}\right)_\tau.
 	\end{eqnarray}
 	
 	According to the last term of right hand side \eqref{R2_v}, using Green's formula  with the fact $\boldsymbol{w}_b=0$ on $\partial \tau$, Cauchy-Schwarz inequality and \eqref{Eqn:3.31}, we have
 	\begin{eqnarray}\label{Eqn:vl2_curlp}
 	\lefteqn{\left(\boldsymbol{curl}_h~\alpha\left(p_{h}-p\right) , \boldsymbol{w}_{b}\right)_\tau
 		=
 		\left(\alpha (p_{h}-p),\mathrm{curl}_h~ \boldsymbol{w}_{b}\right)_\tau} \nonumber\\
 	&\leqslant&
 	\|\alpha\|_{0,\infty} \left\|p_{h}-p\right\|_{0,\tau}\left\|\mathrm{curl}_h~\boldsymbol{w}_{b}\right\|_{0,\tau}
 	\lesssim
 	h_{\tau}^{-1}\|\alpha\|_{0,\infty}\left\|p_{h}-p\right\|_{0,\tau}\|\boldsymbol{w}\|_{0,\tau}.
 	\end{eqnarray}
 	
 	Using \eqref{R2_v} and \eqref{Eqn:vl2_curlp}, we get
 	\begin{equation}\label{v_0}
 	\|\boldsymbol{w}\|_{0,\tau}
 	\lesssim
 	\left\|\boldsymbol{f}_{h}-\boldsymbol{f}\right\|_{0,\tau}
 	+
 	\|\beta\|_{0,\infty}\left\|\boldsymbol{u}_{h}-\boldsymbol{u}\right\|_{0,\tau}
 	+
 	h_{\tau}^{-1}\|\alpha\|_{0,\infty}\left\|p_{h}-p\right\|_{0,\tau}.
 	\end{equation}
 	
 	By the definition of $R_{2}\left( \boldsymbol{u}_{h},p_{h}\right)$ and \eqref{v_0}, we obtain
 	\begin{eqnarray*}
 		\|R_{2}\left(\boldsymbol{u}_{h},p_{h} \right)\|_{0,\tau}
 		&\leqslant&
 		\|\boldsymbol{w}\|_{0,\tau}
 		+
 		\left\|\boldsymbol{f}_{h}-\boldsymbol{f}\right\|_{0,\tau}
 		\\
 		&\lesssim&
 		C_2\Big(
 		\left\|\boldsymbol{f}_{h}-\boldsymbol{f}\right\|_{0,\tau}
 		+
 		\left\|\boldsymbol{u}_{h}-\boldsymbol{u}\right\|_{0,\tau}+h_{\tau}^{-1}\left\|p_{h}-p\right\|_{0,\tau}
 		\Big),
 	\end{eqnarray*}
 	which implies
 	\begin{eqnarray*}
 		\lefteqn{ \sum_{\tau \in \mathcal{T}_h} h_{\tau}^{2}\|R_{2}\left(\boldsymbol{u}_{h},p_{h} \right)\|_{0,\tau}^{2}}
 		\\
 		&& \lesssim
 		C_2\Big(
 		\sum_{\tau \in \mathcal{T}_h}\left(h_{\tau}^{2}\left\|\boldsymbol{f}_{h}-\boldsymbol{f}\right\|_{0,\tau}^{2}
 		+
 		h_{\tau}^{2}\left\|\boldsymbol{u}_{h}-\boldsymbol{u}\right\|_{0,\tau}^{2}
 		+
 		\left\|p_{h}-p\right\|_{0,\tau}^{2}\right)
 		\Big)
 		\\
 		&& \lesssim
 		C_2\Big(
 		\left\| (\boldsymbol{u},p)-(\boldsymbol{u}_{h},p_h)\right\|_{\mathrm{DG}}^{2}
 		+
 		\sum_{\tau \in \mathcal{T}_h} h_{\tau}^{2}\left\|\boldsymbol{f}_{h}-\boldsymbol{f}\right\|_{0,\tau}^{2}
 		\Big).
 	\end{eqnarray*}
 \end{proof}
 
 In the following, we turn to estimate the third term $R_{3}$ of the error estimator.
 
 \begin{lemma}\label{Lemm:R3}
 	Let $(\boldsymbol{u},p)\in\boldsymbol{U}\times\mathbb{Q}$ and $(\boldsymbol{u}_h,p_h)\in\boldsymbol{U}_h\times\mathbb{Q}_h$ be the solutions of  \eqref{Eqn:weak_1}-\eqref{Eqn:weak_2} and \eqref{Eqn:dweak_1}-\eqref{Eqn:dweak_2}, respectively. Then we have
 	\begin{equation}\label{h2R_3}
 	\sum_{\tau \in \mathcal{T}_h} h_{\tau}^{2}\|R_{3}\left(\boldsymbol{u}_{h}\right)\|_{0,\tau}^{2}
 	\lesssim
 	C_1\Big(
 	\left\|\boldsymbol{u}-\boldsymbol{u}_{h}\right\|_{0,\Omega}^{2}+\left\|\boldsymbol{f}-\boldsymbol{f}_{h}\right\|_{0,\Omega}^{2}+\sum_{\tau \in \mathcal{T}_h} h_{\tau}^{2}\|\nabla \cdot\left(\boldsymbol{f}-\boldsymbol{f}_{h}\right)\|_{0,\tau}^{2}
 	\Big),
 	\end{equation}
 	where $C_1$ is a constant depending on   $\|\beta\|_{0,\infty}$.
 \end{lemma}
 \begin{proof}
 	Let $\boldsymbol{f}_{h} \in \mathcal{S}_{h}^{\text {conf }}$, $\chi:=\nabla \cdot\left(\boldsymbol{f}_{h}-\beta \boldsymbol{u}_{h}\right)$ and $\chi_{b}=b_{\tau} \chi$, where $\chi$ is a polynomial on each $\tau \in \mathcal{T}$ and $\chi_{b} \in H_{0}^{1}(\tau) \subset H_{0}^{1}(\Omega)$. Setting $\boldsymbol{v}=\nabla \chi_{b} \in \boldsymbol{U}$ in \eqref{Eqn:weak_2} and with the fact $\chi_b|_{\partial\tau}=0$, we get
 	\begin{equation}\label{Eqn:f-beta}
 	\left(\boldsymbol{f} -\beta \boldsymbol{u}, \nabla \chi_{b}\right)_{\tau}=0.
 	\end{equation}
 	Then, using \eqref{Eqn:3.24}, the Green's formula and \eqref{Eqn:f-beta}, we obtain
 	\begin{eqnarray*}
 		\|\chi\|_{0,\tau}^{2}
 		&\lesssim&
 		\|b_{\tau}^{1 / 2} \chi\|_{0,\tau}^{2}
 		=
 		(\chi,b_\tau\chi)
 		=
 		(\nabla \cdot\left(\boldsymbol{f}_{h}-\beta \boldsymbol{u}_{h}\right) ,\chi_{b})_\tau\\
 		&=&
 		-(\left(\boldsymbol{f}_{h} -\beta \boldsymbol{u}_{h}\right), \nabla \chi_{b})_\tau
 		=
 		\left(\left(\boldsymbol{f}-\boldsymbol{f}_{h}\right)-\beta\left(\boldsymbol{u}-\boldsymbol{u}_{h}\right) ,\nabla \chi_{b}\right)_\tau.
 	\end{eqnarray*}
 	According to the inverse inequalities, there holds
 	\begin{eqnarray*}
 		\|\chi\|_{0,\tau}^{2}
 		&\leqslant&
 		\left(\left\|\boldsymbol{f}_{h}-\boldsymbol{f}\right\|_{0,\tau}
 		+
 		\|\beta\|_{0,\infty}\left\|\boldsymbol{u}_{h}-\boldsymbol{u}\right\|_{0,\tau}\right)\left\|\nabla \chi_{b}\right\|_{0,\tau}\\
 		&\lesssim&
 		C_1h_{\tau}^{-1}\left(\left\|\boldsymbol{f}_{h}-\boldsymbol{f}\right\|_{0,\tau}
 		+
 		\left\|\boldsymbol{u}_{h}-\boldsymbol{u}\right\|_{0,\tau}\right)\|\chi\|_{0,\tau},
 	\end{eqnarray*}
 	which imples
 	\begin{equation}\label{h_tau}
 	h_\tau\|\chi\|\lesssim C_1\left(\left\|\boldsymbol{f}_{h}-\boldsymbol{f}\right\|_{0,\tau}
 	+
 	\left\|\boldsymbol{u}_{h}-\boldsymbol{u}\right\|_{0,\tau}\right).
 	\end{equation}
 	Then combing the definition of $R_3(\boldsymbol{u}_h)$, triangle inequality and \eqref{h_tau}, we get
 	\begin{eqnarray*}
 		h_\tau\|R_{3}\left(\boldsymbol{u}_{h}\right)\|_{0,\tau}
 		&\leqslant&
 		h_\tau\|\chi\|_{0,\tau}
 		+
 		h_\tau\|\nabla \cdot\left(\boldsymbol{f}-\boldsymbol{f}_{h}\right)\|_{0,\tau} \\
 		&\lesssim&
 		C_1\left(\|\boldsymbol{f}_{h}-\boldsymbol{f}\|_{0,\tau}+\|\boldsymbol{u}_{h}-\boldsymbol{u}\|_{0,\tau}\right)
 		+
 		h_\tau\|\nabla \cdot\left(\boldsymbol{f}-\boldsymbol{f}_{h}\right)\|_{0,\tau}.
 	\end{eqnarray*}
 	Summing over all elements $\tau \in \mathcal{T}$, the result \eqref{h2R_3} follows directly.
 \end{proof}
 
 Next, we turn to bound the fourth term $J_{1}$ of the error estimator.
 
 \begin{lemma}\label{Lemm:J1}
 	Let $(\boldsymbol{u},p)\in\boldsymbol{U}\times\mathbb{Q}$ and $(\boldsymbol{u}_h,p_h)\in\boldsymbol{U}_h\times\mathbb{Q}_h$ be the solutions of  \eqref{Eqn:weak_1}-\eqref{Eqn:weak_2} and \eqref{Eqn:dweak_1}-\eqref{Eqn:dweak_2}, respectively. Then we have
 	\begin{eqnarray*}
 		\lefteqn{ \sum_{e \in \mathcal{E}_h} h_{e}\|J_{1}\left(p_{h}\right)\|_{0,e}^{2}}
 		\\
 		&&\lesssim
 		C_2\Big(
 		\sum_{\tau \in \mathcal{T}_h} \|R_{1}\left(\boldsymbol{u}_{h},p_{h} \right)\|_{0,\tau}^{2}
 		+
 		\sum_{\tau \in \mathcal{T}_h} h_{e}^{2}\|R_{2}\left(\boldsymbol{u}_{h},p_{h} \right)\|_{0,\tau}^{2}
 		+
 		\left\| (\boldsymbol{u},p)-(\boldsymbol{u}_{h},p_h)\right\|_{\mathrm{DG}}^{2}
 		\Big),
 	\end{eqnarray*}
 	where $C_2$ is a constant depending on  $\|\alpha\|_{0,\infty}$ and  $\|\beta\|_{0,\infty}$.
 \end{lemma}
 \begin{proof}
 	Without loss of generality, let $e\in  \mathcal{E}_h$ be the common edge of $\tau_1\in  \mathcal{T}_h$ and $\tau_2\in  \mathcal{T}_h$, we give the following definition
 	\begin{equation}\label{Eqn:3-3-e}
 	\boldsymbol{w}_{h} = [[\alpha{p}_{h}]]_{e},
 	\quad \tilde{\boldsymbol{w}}_{b}=b_e \boldsymbol{w}_{h} \in H^1_0(\tau_{1}\cup\tau_2).
 	\end{equation}
 	Applying \eqref{Eqn:3-3-e}, $[[\alpha\mathrm{curl}~\boldsymbol{u}]]_{e}=\boldsymbol{0}$, Green's fromula, \eqref{Eqn:weak_1} and \eqref{Eqn:weak_2}, we get
 	\begin{eqnarray}
 	\nonumber
 	\lefteqn{\|b_e^{\frac{1}{2}} \boldsymbol{w}_h\|^2_{0, e} =\left\langle \boldsymbol{w}_h, b_e\boldsymbol{w}_h\right\rangle_e }\\ \nonumber
 	&=&
 	\left\langle [[ \alpha{p}_{h}]]_e, \tilde{\boldsymbol{w}}_{b}\right\rangle_e  \\  \nonumber
 	&=&
 	\left\langle [[  \alpha (p_{h}-\mathrm{curl}~\boldsymbol{u})  ]]_{e}, \tilde{\boldsymbol{w}}_{b}\right\rangle_e  \\ \nonumber
 	&=&
 	\sum\limits_{\tau\in \omega_e}\left[\left(
 	\boldsymbol{curl}\left(
 	\alpha (p_{h}-\mathrm{curl}~\boldsymbol{u})
 	\right), \tilde{\boldsymbol{w}}_{b}
 	\right)_\tau
 	-
 	\left(
 	\alpha (p_{h}-\mathrm{curl}~\boldsymbol{u}), \mathrm{curl}~\tilde{\boldsymbol{w}}_{b}
 	\right)_\tau\right]   \\  \nonumber
 	&=&
 	\sum\limits_{\tau\in \omega_e}\left[\left( \boldsymbol{curl}~
 	\alpha p_{h}-\boldsymbol{curl}~\alpha\mathrm{curl}~\boldsymbol{u}
 	+
 	\boldsymbol{f}-\boldsymbol{f}
 	+
 	\beta\boldsymbol{u}-\beta\boldsymbol{u}
 	+
 	\beta\boldsymbol{u}_h-\beta\boldsymbol{u}_h, \tilde{\boldsymbol{w}}_{b}\right)_\tau \right.   \\  \nonumber
 	&&
 	\left. -\left(\alpha (p_{h}-\mathrm{curl}~\boldsymbol{u}), \mathrm{curl}~\tilde{\boldsymbol{w}}_{b}\right)_\tau\right]
 	\\  \nonumber
 	&=&
 	\sum\limits_{\tau\in \omega_e}\left[ \left(
 	\boldsymbol{curl}~\alpha p_h + \beta\boldsymbol{u}_h	- \boldsymbol{f}, \tilde{\boldsymbol{w}}_{b}\right)_\tau
 	+
 	\left(\beta\boldsymbol{u}-\beta\boldsymbol{u}_h, \tilde{\boldsymbol{w}}_{b}\right)_\tau \right.   \\  \label{Eqn:3.3.13h}
 	&&
 	\left. -\left(\alpha (p_{h}-\mathrm{curl}~\boldsymbol{u}), \mathrm{curl} ~\tilde{\boldsymbol{w}}_{b}\right)_\tau\right].
 	\end{eqnarray}
 	
 	Using \eqref{Eqn:3.32}, \eqref{Eqn:3.3.13h}, Cauchy-Schwarz inequality, \eqref{Eqn:3.33} and \eqref{Eqn:3.34}, there holds
 	\begin{eqnarray*}
 		\lefteqn{\|\boldsymbol{w}_h\|^2_{0, e}\lesssim\|b_e^{\frac{1}{2}} \boldsymbol{w}_h\|^2_{0, e}
 		}
 		\\
 		&\lesssim&
 		\sum\limits_{\tau\in \omega_e}\left[ \left(
 		\boldsymbol{curl}~\alpha p_{h} + \beta\boldsymbol{u}_h - \boldsymbol{f}, \tilde{\boldsymbol{w}}_{b}\right)_\tau
 		+
 		\left(\beta\boldsymbol{u}-\beta\boldsymbol{u}_h, \tilde{\boldsymbol{w}}_{b}\right)_\tau \right.   \\
 		&&
 		\left. -\left(\alpha (p_{h}-\mathrm{curl}~\boldsymbol{u}), \mathrm{curl}~ \tilde{\boldsymbol{w}}_{b}\right)_\tau\right]
 		\\
 		&\lesssim&
 		\sum\limits_{\tau\in \omega_e}\left( \left\|\boldsymbol{curl}~\alpha p_{h} + \beta\boldsymbol{u}_h - \boldsymbol{f}\right\|_{0, \tau}\cdot\left\|\tilde{\boldsymbol{w}}_{b}\right\|_{0, \tau}
 		+
 		\|\beta\|_{0,\infty}\left\|\boldsymbol{u}-\boldsymbol{u}_h\right\|_{0, \tau}\cdot\left\|\tilde{\boldsymbol{w}}_{b}\right\|_{0, \tau} \right. \\
 		&&
 		\left. +\|\alpha\|_{0,\infty}\left\|p_{h}-\mathrm{curl}~\boldsymbol{u}\right\|_{0, \tau}\cdot\left\|\mathrm{curl}~ \tilde{\boldsymbol{w}}_{b}\right\|_{0, \tau}\right)
 		\\
 		&\lesssim&
 		C_2\sum\limits_{\tau\in \omega_e}\left( \left\|\boldsymbol{curl}~\alpha p_{h} + \beta\boldsymbol{u}_h - \boldsymbol{f}\right\|_{0, \tau}\cdot h^{\frac{1}{2}}_e\left\|\boldsymbol{w}_{h}\right\|_{0, e} \right. \\
 		&&
 		\left. +\left\|\boldsymbol{u}-\boldsymbol{u}_h\right\|_{0, \tau}\cdot h^{\frac{1}{2}}_e\left\|\boldsymbol{w}_{h}\right\|_{0, e}
 		+
 		\left\|p_{h}-\mathrm{curl}~\boldsymbol{u}\right\|_{0, \tau}\cdot h^{-\frac{1}{2}}_e\left\|\boldsymbol{w}_{h}\right\|_{0, e}\right)
 		\\
 		&\lesssim&
 		C_2\sum\limits_{\tau\in \omega_e}\left(
 		h^{\frac{1}{2}}_e \left\|\boldsymbol{curl}~\alpha p_{h} + \beta\boldsymbol{u}_h - \boldsymbol{f}\right\|_{0, \tau}+h^{\frac{1}{2}}_e \left\|\boldsymbol{u}-\boldsymbol{u}_h\right\|_{0, \tau}  \right. \\
 		&&
 		\left.  +h^{-\frac{1}{2}}_e\left\|p_{h}-\mathrm{curl}_h~\boldsymbol{u}_h\right\|_{0, \tau}
 		+h^{-\frac{1}{2}}_e\left\|\mathrm{curl}~\boldsymbol{u}-\mathrm{curl}_h~\boldsymbol{u}_h\right\|_{0, \tau}
 		\right)\cdot \left\|\boldsymbol{w}_{h}\right\|_{0, e},
 	\end{eqnarray*}
 	which imples
 	\begin{eqnarray}
 	\nonumber
 	\left\|\boldsymbol{w}_h\right\|_{0, e}
 	&\lesssim&
 	C_2\sum\limits_{\tau\in \omega_e}\left( h^{\frac{1}{2}}_e \left\|\boldsymbol{f}-\boldsymbol{curl}~\alpha p_h-\beta\boldsymbol{u}_h\right\|_{0, \tau}
 	+
 	h^{\frac{1}{2}}_e \left\|\boldsymbol{u}-\boldsymbol{u}_h\right\|_{0, \tau} \right.  \\   \label{Eqn:3.41}
 	&&
 	\left. +h^{-\frac{1}{2}}_e\left\|p_{h}-\mathrm{curl}_h~\boldsymbol{u}_h\right\|_{0, \tau}
 	+h^{-\frac{1}{2}}_e\left\|\mathrm{curl}~\boldsymbol{u}-\mathrm{curl}_h~\boldsymbol{u}_h\right\|_{0, \tau}\right).
 	\end{eqnarray}

 	Applying \eqref{Eqn:3-3-e} and \eqref{Eqn:3.41} results in
 	\begin{eqnarray*}
 		\lefteqn{h_e\|J_1( p_h)\|^2_{0, e}
 			=
 			h_e\|[[\alpha p_h]]_\tau\|^2_{0, e}
 			=
 			h_e\left\|\boldsymbol{w}_h\right\|^2_{0, e}
 		}
 		\\
 		&& \lesssim
 		C_2\Big(
 		\sum\limits_{\tau\in \omega_e}\left(
 		\| R_1(\boldsymbol{u}_h,p_h) \|_{0, \tau}
 		+
 		h_e^2\| R_2(\boldsymbol{u}_h,p_h) \|_{0, \tau} \right) + \|\left(p-p_{h}, \boldsymbol{u}-\boldsymbol{u}_{h}\right)\|_{\mathrm{DG}}^{2}
 		\Big).
 	\end{eqnarray*}
 \end{proof}
 
 At last, we turn to bound the fifth term $J_{2}$ of the error estimator.
 
 \begin{lemma}\label{lem:J2}
 	Let $(\boldsymbol{u},p)\in\boldsymbol{U}\times\mathbb{Q}$ and $(\boldsymbol{u}_h,p_h)\in\boldsymbol{U}_h\times\mathbb{Q}_h$ be the solutions of  \eqref{Eqn:weak_1}-\eqref{Eqn:weak_2} and \eqref{Eqn:dweak_1}-\eqref{Eqn:dweak_2}, respectively. Then we have
 	\begin{equation*}
 	\sum_{e \in \mathcal{E}_h} h_{e}\|J_{2}\left(\boldsymbol{u}_{h}\right)\|_{0,e}^{2}
 	\lesssim
 	C_2\Big(
 	\|\boldsymbol{u}-\boldsymbol{u}_{h}\|_{0,\Omega}^{2}
 	+
 	\|\boldsymbol{f}-\boldsymbol{f}_{h}\|_{0,\Omega}^{2}
 	+
 	\sum_{\tau \in \mathcal{T}_h} h_{\tau}^{2}
 	\|\nabla \cdot\left(\boldsymbol{f}-\boldsymbol{f}_{h}\right)\|_{0,\Omega}^{2}
 	\Big),
 	\end{equation*}
 	where $C_1$ is a constant depending on   $\|\beta\|_{0,\infty}$.
 \end{lemma}
 \begin{proof}
 	Without loss of generality, let $e \in \mathcal{E}_h$ be the common edge of $\tau_{1}$ and $\tau_{2}$. Setting  $\phi=[(\boldsymbol{f}_{h}-\beta \boldsymbol{u}_{h})]_e$ and $\tilde{\phi}_{b}=b_{e} \phi$, then according to \eqref{Eqn:3.26}, there holds
 	\begin{eqnarray}\label{f_h-betauh}
 	\|\phi\|_{0,e}^{2}
 	&\lesssim&
 	\|b_{e}^{1 / 2} \phi\|_{0,e}^{2}
 	=
 	(\phi,b_e\phi)
 	=
 	\langle [ \left(\boldsymbol{f}_{h}-\beta \boldsymbol{u}_{h}\right)]_e, \tilde{\phi}_{b} \rangle_e \nonumber \\
 	&=&
 	\sum_{i=1,2}\left( ( \nabla \cdot\left(\boldsymbol{f}_{h}-\beta \boldsymbol{u}_{h}\right), \tilde{\phi}_{b})_{\tau_{i}}
 	+
 	(\boldsymbol{f}_{h}-\beta  \boldsymbol{u}_{h}, \nabla \tilde{\phi}_{b})_{\tau_{i}}  \right),
 	\end{eqnarray}
 	where we have used the fact that $\tilde{\phi}_{b} \in H_{0}^{1}\left(\tau_{1} \cup \tau_{2}\right)$.
 	Applying  the definition of $R_{3}\left(\boldsymbol{u}_{h}\right)$ and \eqref{Eqn:3.27} to the first term $( \nabla \cdot\left(\boldsymbol{f}_{h}-\beta \boldsymbol{u}_{h}\right), \tilde{\phi}_{b})_{\tau_{i}}$ of \eqref{f_h-betauh} gives
 	\begin{eqnarray}
 	(\nabla \cdot\left(\boldsymbol{f}_{h}-\beta \boldsymbol{u}_{h}\right), \tilde{\phi}_{b})_{\tau_{i}}
 	&=&
 	(R_{3}\left(\boldsymbol{u}_{h}\right), \tilde{\phi}_{b})_{\tau_{i}}
 	+
 	(\nabla \cdot\left(\boldsymbol{f}_{h}-\boldsymbol{f}\right), \tilde{\phi}_{b} )_{\tau_{i}}
 	\nonumber\\
 	&\leqslant&
 	\left(\|R_{3}\left(\boldsymbol{u}_{h}\right)\|_{0,\tau_{i}}+\|\nabla \cdot\left(\boldsymbol{f}-\boldsymbol{f}_{h}\right)\|_{0,\tau_{i}}\right) \cdot\|\tilde{\phi}_{b}\|_{0,\tau_{i}}
 	\nonumber\\
 	&\lesssim& h_{e}^{\frac{1}{2}}\left(\|R_{3}\left(\boldsymbol{u}_{h}\right)\|_{0,\tau_{i}}
 	+
 	\|\nabla \cdot\left(\boldsymbol{f}-\boldsymbol{f}_{h}\right)\|_{0,\tau_{i}}\right) \cdot\|\phi\|_{0,e}.
 	\end{eqnarray}
 	
 	Using the Cauchy-Schwarz inequality and \eqref{Eqn:3.28} to the second term of  \eqref{f_h-betauh} gives
 	\begin{eqnarray}
 	\lefteqn{\sum_{i=1,2} \left(\boldsymbol{f}_{h}-\beta \boldsymbol{u}_{h} , \nabla \tilde{\phi}_{b} \right)_{\tau_{i}}
 		=
 		\sum_{i=1,2} \left(\boldsymbol{f}_{h}-\boldsymbol{f}-\beta\left(\boldsymbol{u}_{h}-\boldsymbol{u}\right) , \nabla \tilde{\phi}_{b}\right)_{\tau_{i}}} \nonumber\\
 	&\lesssim& C_1\sum_{i=1,2}\left(\left\|\boldsymbol{f}-\boldsymbol{f}_{h}\right\|_{0,\tau_{i}}+\left\|\boldsymbol{u}-\boldsymbol{u}_{h}\right\|_{0,\tau_{i}}\right)\cdot\|\nabla \tilde{\phi}_{b}\|_{0,\tau_{i}}\nonumber\\
 	&\lesssim&
 	C_1\sum_{i=1,2} h_{e}^{-1 / 2}\left(\left\|\boldsymbol{f}-\boldsymbol{f}_{h}\right\|_{0,\tau_{i}}+\left\|\boldsymbol{u}-\boldsymbol{u}_{h}\right\|_{0,\tau_{i}}\right) \cdot\|\phi\|_{0,e},
 	\end{eqnarray}
 	where we have used $	\left(\boldsymbol{f}-\beta \boldsymbol{u}, \nabla \tilde{\phi}_{b}\right)_{\tau_{1} \cup \tau_{2}}=0$ when we set $\boldsymbol{v}=\nabla \tilde{\phi}_{b} \in \boldsymbol{U}$ in \eqref{Eqn:weak_2}.
 	
 	According to the assumptions $\boldsymbol{f} \in \boldsymbol{H}(\operatorname{div}, \Omega)$ and $\boldsymbol{f}_{h} \in \mathcal{S}_{h}^{\text {conf }} $, there holds
 	\begin{equation}\label{f-f_h}
 	[\boldsymbol{f}-\boldsymbol{f}_{h}]_{e}=0.
 	\end{equation}
 	
 	Due to the definition of $J_{2}\left(\boldsymbol{u}_{h}\right)$ and  \eqref{f_h-betauh}-\eqref{f-f_h}, we arrive at
 	\begin{eqnarray*}
 		\|J_{2}\left(\boldsymbol{u}_{h}\right)\|_{0,e}
 		&=&
 		\|\phi\|_{0,e}+\|[\left(\boldsymbol{f}-\boldsymbol{f}_{h}\right)]_e\|_{0,e}
 		\\
 		&\lesssim&
 		C_1\sum_{i=1,2} \Big(
 		h_{e}^{1 / 2}\left(\|R_{3}\left(\boldsymbol{u}_{h}\right)\|_{0,\tau_{i}}
 		+
 		\|\nabla \cdot\left(\boldsymbol{f}-\boldsymbol{f}_{h}\right)\|_{0,\tau_{i}}\right)
 		\Big.
 		\\
 		&&\Big.+ h_{e}^{-1 / 2}\left(\left\|\boldsymbol{f}-\boldsymbol{f}_{h}\right\|_{0,\tau_{i}}
 		+
 		\left\|\boldsymbol{u}-\boldsymbol{u}_{h}\right\|_{0,\tau_{i}}\right)
 		\Big).
 	\end{eqnarray*}
 	Hence, combining Lemma \ref{Lemm:R3}, we have
 	\begin{eqnarray*}
 		\sum_{e} h_{e}\|J_{2}
 		\left(\boldsymbol{u}_{h}\right)\|_{0,e}^{2}
 		&\lesssim&
 		C_1\sum_{\tau}\left(h_{\tau_{i}}^{2}\left(
 		\|R_{3}\left(\boldsymbol{u}_{h}\right)\|_{0,\tau_{i}}^2
 		+
 		\|\nabla \cdot\left(\boldsymbol{f}-\boldsymbol{f}_{h}\right)\|_{0,\tau_{i}}^2\right) \right.\\
 		&&\left.+\left\|\boldsymbol{f}-\boldsymbol{f}_{h}\right\|_{0,\tau_{i}}^2
 		+
 		\left\|\boldsymbol{u}-\boldsymbol{u}_{h}\right\|_{0,\tau_{i}}^2\right) \\
 		&\lesssim&
 		C_1\Big(
 		\|\boldsymbol{u}-\boldsymbol{u}_{h}\|_{0,\Omega}^{2}
 		+\|\boldsymbol{f}-\boldsymbol{f}_{h}\|_{0,\Omega}^{2}+h^2_\tau\|\nabla \cdot\left(\boldsymbol{f}-\boldsymbol{f}_{h}\right)\|_{0,\Omega}^{2}
 		\Big).
 	\end{eqnarray*}
 \end{proof}
 
 Next, we present the proof of Theorem \ref{The:eff}.
 \begin{proof}[Proof of Theorem \ref{The:eff}.]\
 	The following result follows immediately by a direct application of Lemmas \ref{Lemm:R1}-\ref{lem:J2},
 	\begin{equation*}
 	\kappa\sum_{e\in\mathcal{T}_h}\sum_{e\in\partial\tau}h^{-1}_e\left\|J_3(\boldsymbol{u}_h)\right\|^2_{0, e}
 	\lesssim
 	\|(\boldsymbol{u},p)-(\boldsymbol{u}_h,p_h)\|_{DG}.
 	\end{equation*}
 \end{proof}

 \section{Numerical experiment}\label{sec:6}
 In this section, we report some experiments to show the performance of the error indicator and the adaptive algorithm AMIPDG.
 We carry out these numerical experiments by using the MATLAB software package iFEM \cite{ChenLiFEM}. In Experiments \ref{Exa:ex1} and \ref{Exa:ex3}, we take $p=\mathrm{curl}~\boldsymbol{u}$.
 
 In Example \ref{Exa:ex1}, we discuss the influence of the penalty parameter $\kappa$ on the error both in  $L^2$ and $\|\cdot\|_{DG}$ norms, and observe the dependency of the condition number of stiffness matrix on $\kappa$. At last, we verify the reliability and efficiency of the constructed error indicator \eqref{eta}.
 
 \begin{example}\label{Exa:ex1}
 	Let $\Omega:=[-1,1] \times[-1,1]$, we construct the following analytical solution of the model \eqref{Equ:1.1}-\eqref{Equ:1.2}:
 	$$
 	\boldsymbol{u}=\left(\begin{array}{c}
 	\cos (\pi x) \sin (\pi y) \\
 	-\cos (\pi y) \sin (\pi x)
 	\end{array}\right)
 	$$
 	with coefficients
 	$$
 	\alpha=1, \quad \beta=1,
 	$$
 	which corresponds to a right hand source term
 	$$
 	\boldsymbol{f}(x, y)=\left(\begin{array}{c}
 	\left(2 \pi^{2}-1\right) \cos (\pi x) \sin (\pi y) \\
 	-\left(2 \pi^{2}-1\right) \cos (\pi y) \sin (\pi x)
 	\end{array}\right).
 	$$
 	It is easy to see that the solution $\boldsymbol{u}$ satisfies the boundary condition $\boldsymbol{u} \cdot \boldsymbol{t}=0$ on $\partial \Omega$.
 \end{example}
 
 In this example, we get an uniform meshes  by partitioning the $x-$ and $y-$ into equally distributed $M(M\geq 2)$ subintervals,  and then dividing one square  into two triangles.  Let $h=1/M$ be mesh sizes for different triangular meshes.
 Firstly, we fixed mesh with $h=1/32$ and report the error estimates in both  $L^2$ and $\|\cdot\|_{DG}$  norm for different penalty parameters $ \kappa = 1, 50, 100, 150$ and $200$ in Table \ref{Tab_111}. We note that $\left\|\boldsymbol{u}-\boldsymbol{u}_{h}\right\|_{0}$ increases slightly as the penalty parameter $\kappa$ increases. On the contrary, $\|\left(p-p_{h}, \boldsymbol{u}-\boldsymbol{u}_{h}\right)\|_{\mathrm{DG}}$ decreases slightly as $\kappa$ increases.
 \begin{table}[ht]
 	\centering\caption{The errors in both $L^2$ and $\|\cdot\|_{DG}$ norms with $h=1/32$.}\label{Tab_111}
 	\begin{tabular}{ccc}
 		\hline
 		$\kappa $	& $\|\boldsymbol{u}-\boldsymbol{u}_h\|_0$ & $\|\left(p-p_{h}, \boldsymbol{u}-\boldsymbol{u}_{h}\right)\|_{\mathrm{DG}}$   \\ \hline
 		1	& 2.00399e-02    & 1.51333e-01    \\
 		50  & 2.00391e-02    & 1.46739e-01     \\
 		100 & 2.00397e-02    & 1.46733e-01     \\
 		150 & 2.00400e-02    & 1.46732e-01     \\
 		200 & 2.00401e-02   & 1.46732e-01     \\
 		\hline
 	\end{tabular}
 \end{table}

 Next, we also use  fixed mesh with $h=1/32$,  and observe the influence of different $\kappa$ on the condition numbers of stiffness matrices  in Table \ref{Tab_2}. It is easy to see that the condition numbers of stiffness matrices increase with the increase of penalty parameters $\kappa$.
 
 \begin{table}[ht]
 	\centering\caption{Condition number of stiffness matrices with different  $\kappa$.}\label{Tab_2}
 	\begin{tabular}{ccccccc}
 		\hline
 		$\kappa$	& 1  & 50    & 100 & 150 & 200 \\\hline
 		Cond &3.40461e+05 &	1.11946e+07 & 2.42288e+07&  3.73104e+07 & 5.04046e+07\\
 		\hline
 	\end{tabular}
 \end{table}
 As a way to balance, in the following numerical tests, we always choose $\kappa=50$.
 
 Now, we verify the reliability and efficiency of the error estimate by comparing $\eta\left(\boldsymbol{u}_{h}, {p}_{h} ; \mathcal{T}_{h}\right)$  with $\|(p-p_{h}, \boldsymbol{u}-\boldsymbol{u}_{h})\|_{\mathrm{DG}}$. The numerical results are given in Table \ref{Tab:eta}  which also provide the values for the effectivity index $\sigma=\|\left(p-p_{h}, \boldsymbol{u}-\boldsymbol{u}_{h}\right)\|_{\mathrm{DG}}/\eta\left(\boldsymbol{u}_{h}, {p}_{h} ; \mathcal{T}_{h}\right)$. We observe that the convergence rate of $\|\left(p-p_{h}, \boldsymbol{u}-\boldsymbol{u}_{h}\right)\|_{\mathrm{DG}}$ is first order and the effectivity index  $\sigma \approx$ 0.286. This shows that the error indicator is effective and reliable, which is \eqref{Eqn:goodind}.
 \noindent
 \begin{table}[ht]
 	\centering\caption{Rate of convergence of the $\|\left(p-p_{h}, \boldsymbol{u}-\boldsymbol{u}_{h}\right)\|_{\mathrm{DG}}$ and $\eta\left(\boldsymbol{u}_{h}, {p}_{h} ; \mathcal{T}_{h}\right)$ on uniform triangular meshes. }\label{Tab:eta}
 	\label{table-2}\vskip 0.1cm
 	\begin{tabular}{{p{1cm}p{2cm}p{2cm}p{2cm}p{2cm}p{2cm}}}\hline
 		\multirow{2}{*} {$h$}& \multicolumn{2}{c}{$\|\left(p-p_{h}, \boldsymbol{u}-\boldsymbol{u}_{h}\right)\|_{\mathrm{DG}}$} &  \multicolumn{2}{c}{$\eta\left(\boldsymbol{u}_{h}, {p}_{h} ; \mathcal{T}_{h}\right)$} &
 		\multicolumn{1}{c}{index  }
 		\\\cline { 2 - 3 } \cline { 4 - 5 }
 		& \multicolumn{1}{c}{Error}  & \multicolumn{1}{c}{order}  & \multicolumn{1}{c}{Error} &\multicolumn{1}{c}{order} & \multicolumn{1}{c}{$\sigma$} \\ \hline
 		1/16 & \multicolumn{1}{c}{2.93259E-01} & \multicolumn{1}{c}{N/A} & \multicolumn{1}{c}{1.02621E-00} & \multicolumn{1}{c}{N/A} & \multicolumn{1}{c}{0.286}\\
 		1/32 & \multicolumn{1}{c}{1.46739E-01} & \multicolumn{1}{c}{0.9989} & \multicolumn{1}{c}{5.13690E-01} & \multicolumn{1}{c}{0.9984} & \multicolumn{1}{c}{0.286}\\
 		1/64 & \multicolumn{1}{c}{7.33830E-02} & \multicolumn{1}{c}{0.9997} & \multicolumn{1}{c}{2.56920E-01} & \multicolumn{1}{c}{0.9996} & \multicolumn{1}{c}{0.286}\\
 		1/128 & \multicolumn{1}{c}{3.66932E-02} & \multicolumn{1}{c}{0.9999} & \multicolumn{1}{c}{1.28470E-02} & \multicolumn{1}{c}{0.9999} & \multicolumn{1}{c}{0.286}\\
 		1/256 & \multicolumn{1}{c}{1.83468E-02} & \multicolumn{1}{c}{1.0000} & \multicolumn{1}{c}{6.42365E-02} & \multicolumn{1}{c}{1.0000} & \multicolumn{1}{c}{0.286}  \\
 		\hline
 	\end{tabular}
 \end{table}
 
 Noting that we only consider uniform meshes and the constant coefficients  in Example \ref{Exa:ex1}. Next we test adaptive meshes and the jump coefficients. Our adaptive cycle can be implemented by the following algorithm:
 \begin{algorithm}  	
 	\caption{An adaptive mixed interior penalty discontinuous method (AMIPDG) cycle} \label{ALG1}
 	\begin{algorithmic}
 		\STATE {	\textbf{Input} initial triangulation $\mathcal{T}_0$; data $\boldsymbol{f}$; tolerance tol; marking parameter $\theta\in(0,1)$.}
 		
 		\STATE{   	\textbf{Output} a triangulation $\mathcal{T}_J$; MIPDG solution $(\boldsymbol{u}_J,p_J)$.}
 		
 		\STATE{ 	$\eta=1;k=0;$}
 		
 		\STATE{\textbf{while} $\eta\geq tol$}
 		
 		\STATE{	~~~~\textbf{SOLVE} solve  discrete varational problem \eqref{Eqn:dweak_1}-\eqref{Eqn:dweak_2} on $\mathcal{T}_k$ to get the solution $(\boldsymbol{u}_k,p_k)$;}
 		
 		\STATE{	~~~~\textbf{ESTIMATE} compute the posterior error estimator $\eta=\eta(\boldsymbol{u}_k,p_k,\mathcal{T}_k)$ by using \eqref{eta};}
 		
 		\STATE{	~~~~\textbf{MARK} seek a minimum cardinality $\mathcal{M}_{k} \subset \mathcal{T}_{k}$ such that
 			\begin{equation*}\label{mark}
 			\eta^{2}\left(\boldsymbol{u}_{k},p_k, \mathcal{M}_{k}\right) \geq \theta \eta^{2}\left(\boldsymbol{u}_{k},p_k, \mathcal{T}_{k}\right);
 			\end{equation*}}
 		
 		\STATE{	~~~~\textbf{REFINE} Bisect elements in $\mathcal{M}_k$ and the neighboring elements to form a conforming $\mathcal{T}_{k+1}$;}
 		
 		\STATE{	~~~~$k=k+1$; }
 		
 		\STATE{\textbf{end}}
 		
 		\STATE{$u_J=u_k;~p_J=p_k;~\mathcal{T}_J=\mathcal{T}_k;$}
 	\end{algorithmic}
 \end{algorithm}
 
 \begin{example}\label{Exa:ex3}
 	Let $\Omega:=[-1,1] \times[-1,1]$, we construct the following analytical solution of the model \eqref{Equ:1.1}-\eqref{Equ:1.2}
 	$$
 	\boldsymbol{u}=\left(\begin{array}{c}
 	\frac{y\left(x^{2}-1\right)\left(y^{2}-1\right)}{x^{2}+y^{2}+0.02} \\
 	\frac{-x\left(x^{2}-1\right)\left(y^{2}-1\right)}{x^{2}+y^{2}+0.02}
 	\end{array}\right),
 	$$
 	with the jump coefficients
 	$$
 	\begin{array}{ll}
 	\alpha=1.0, \beta=1.0, & \text { on } \Omega_{1}, \\
 	\alpha=1.0, \beta=100, & \text { on } \Omega \backslash \Omega_{1},
 	\end{array}
 	$$
 	where $\Omega_{1}=(-0.5,0.5)^{2}$(see the left of Figure \ref{fig:ex3_mesh_figure}). Note that the solution $\boldsymbol{u}$ satisfies the condition $ \boldsymbol{u}\cdot\boldsymbol{t}=0$ on $\partial \Omega$.
 \end{example}
 
 Here, we can observe that $\boldsymbol{u}$ has a relatively large change at $(0,0)$. Figure \ref{fig:ex3_figure} shows the contours of the exact solution.
 \begin{figure}[ht]
 	\begin{minipage}[t]{0.48\linewidth}
 		\centering
 		\includegraphics[height=5.2cm,width=6.8cm]{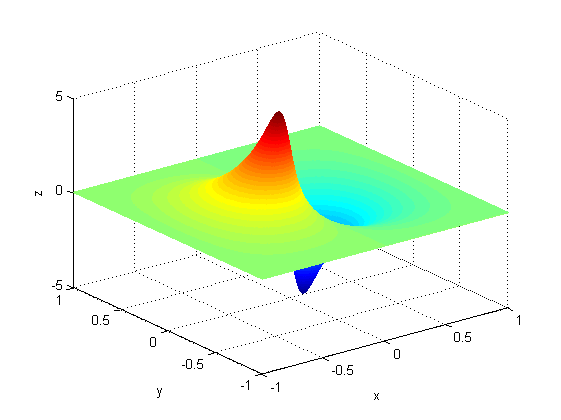}
 	\end{minipage}
 	\begin{minipage}[t]{0.48\linewidth}
 		\centering
 		\includegraphics[height=5.2cm,width=6.8cm]{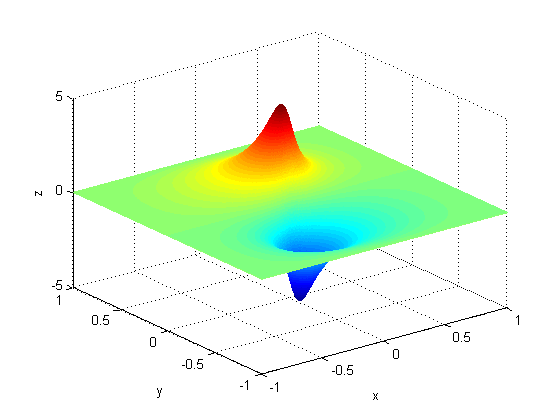}
 	\end{minipage}
 	\caption{Left: the first component of the analytical solution. Right: the second component of the analytical solution.}\label{fig:ex3_figure}
 \end{figure}
 
 We get an initial mesh $\mathcal{T}_0$ by partitioning the $x$- and $y$-axes into equally distributed eight subintervals and then dividing one square into two triangles, see the middle of Figure \ref{fig:ex3_mesh_figure}.
 The right of Figure \ref{fig:ex3_mesh_figure} shows an adaptively refined mesh with marking parameter $\theta=0.5$ after $k=8$, and we can see that the grid is locally refined near both the origin and at $\partial\Omega_1$.
 
 \begin{figure}[ht]
 	\subfigure{}
 	\begin{minipage}[t]{4cm}
 		\centering
 		\includegraphics[height=3cm,width=3.4cm]{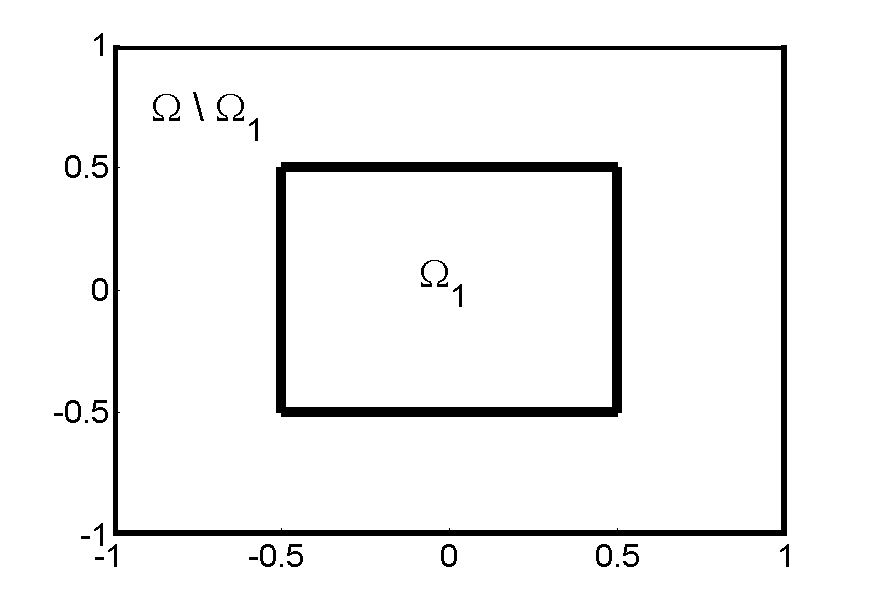}
 	\end{minipage}
 	\subfigure{}
 	\begin{minipage}[t]{4cm}
 		\centering
 		\includegraphics[height=3cm,width=4.0cm]{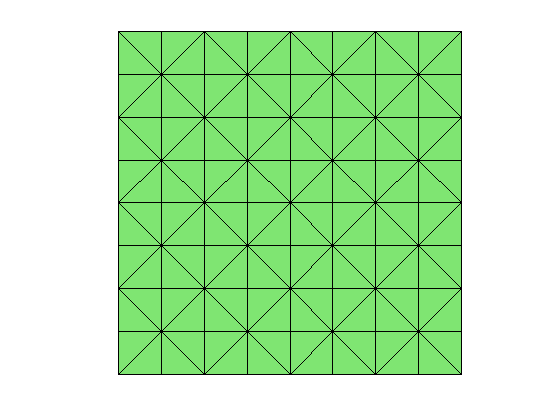}
 	\end{minipage}
 	\subfigure{}
 	\begin{minipage}[t]{4cm}
 		\centering
 		\includegraphics[height=3cm,width=4.0cm]{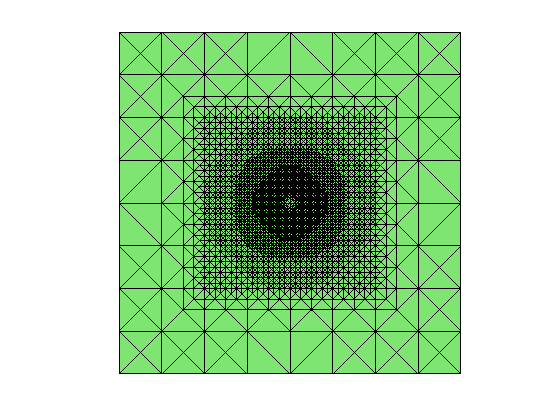}
 	\end{minipage}
 	\caption{Left: regional diagram. Middle: the initial mesh with 512 DoFs. Right: the adaptive mesh($\theta=0.5$) with 20416  DoFs after 8 refinements.}\label{fig:ex3_mesh_figure}
 \end{figure}

 Figure \ref{fig:ex3_res} shows the curves of $\ln N-\ln\|(p-p_k,\boldsymbol{u}-\boldsymbol{u}_k)\|_{DG}$ for different marking parameters $\theta=0.3, 0.5$ and $0.7$, where $N$ is the number of degrees of freedom. The curves indicate the convergence and the quasi-optimality of the adaptive algorithm AMIPDG of the energy error $\|\left(p-p_{h}, \boldsymbol{u}-\boldsymbol{u}_{h}\right)\|_{\mathrm{DG}}$, i.e.
 \begin{equation*}
 \|(p-p_k,\boldsymbol{u}-\boldsymbol{u}_k)\|_{DG}\lesssim  N^{-1/2}.
 \end{equation*}
 From these curves, it seems that the convergence rate are robust for $\theta$ changing from $0.3$ to $0.7$.

 \begin{figure}[ht]
 	\centering
 	\includegraphics[height=8cm,width=10cm]{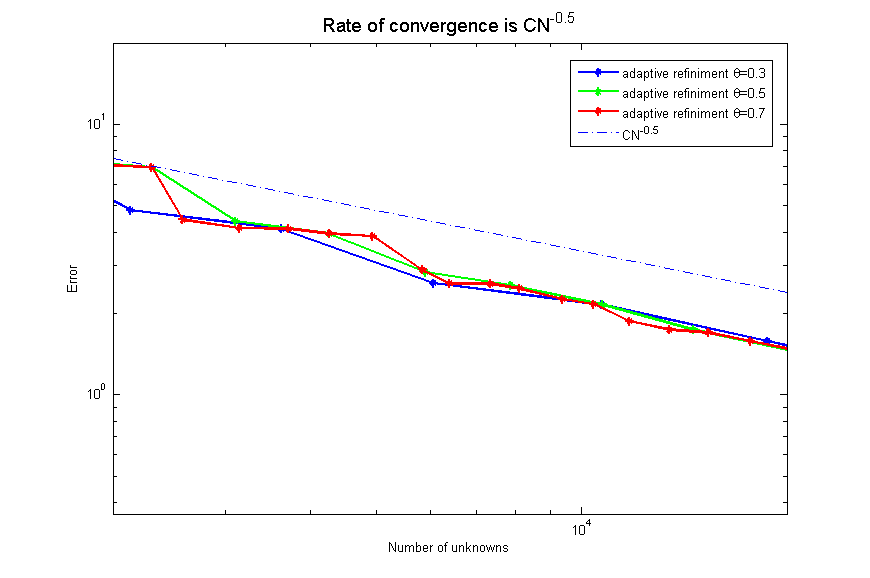}
 	\caption{Quasi optimality of the AMIPDG of the error $\|\left(p-p_{h}, \boldsymbol{u}-\boldsymbol{u}_{h}\right)\|_{\mathrm{DG}}$ with different marking parameters $\theta$.}\label{fig:ex3_res}
 \end{figure}
 
 The examples above are considered convex domain and the jump coefficients. Finally we consider variable coefficients, `L-shaped' domain and unknown exact solution.
 
 \begin{example}\label{Exa:ex2}
 	Let domain $\Omega=(-1,1)^{2}/([0,1)\times[0,1))$
 	and let the variable coefficients
 	$$
 	\alpha=\frac{1}{1+x^{2}+y^{2}}, \quad \beta=\left(\begin{array}{cc}
 	1+x^{2} & x y \\
 	x y & 1+y^{2}
 	\end{array}\right).
 	$$
 	We set the homogeneous Dirichlet boundary condition $\boldsymbol{u}\cdot\boldsymbol{t}=0$
 	on $\partial\Omega$, the source $f=(\frac{1}{x^2+y^2+0.01},\frac{1}{x^2+y^2+0.01})$.
 \end{example}
 
 We get an initial mesh $\mathcal{T}_0$ by partitioning the $x$- and $y$-axes into equally distributed eight subintervals and then dividing one square into two triangles, see the left of Figure \ref{fig:ex2_mesh_figure}.
 The right of Figure \ref{fig:ex2_mesh_figure} shows an adaptively refined mesh with marking parameter- $\theta=0.5$ after $k=8$. The grid is locally refined near the origin.
 
 \begin{figure}[h]
 	\begin{minipage}[t]{0.48\linewidth}
 		\centering
 		\includegraphics[height=5cm,width=6cm]{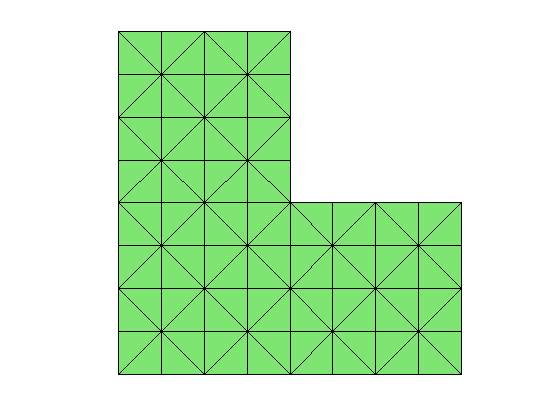}
 	\end{minipage}
 	\begin{minipage}[t]{0.48\linewidth}
 		\centering
 		\includegraphics[height=5cm,width=6cm]{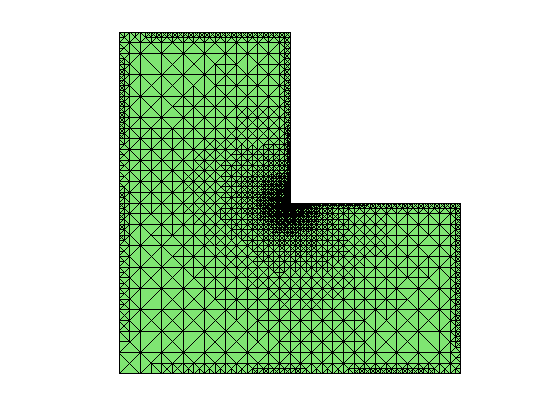}
 	\end{minipage}
 	\caption{Left: the initial mesh with 384 DoFs. Right: the adaptive mesh($\theta=0.5$) with 16032  DoFs after 8 refinements.}\label{fig:ex2_mesh_figure}
 \end{figure}
 
 The Figure \ref{fig:ex2_res} shows the curves of $\ln N -\ln\eta\left(\boldsymbol{u}_{k}, {p}_{k} ; \mathcal{T}_{k}\right)$ for parameters $\theta=0.3,0.5,0.7$. The curves indicate the convergence and the quasi-optimality of the adaptive algorithm AMIPDG of  $\eta\left(\boldsymbol{u}_{k}, {p}_{k} ; \mathcal{T}_{k}\right)$.
 
 \begin{figure}[ht]
 	\centering
 	\includegraphics[height=8cm,width=10cm]{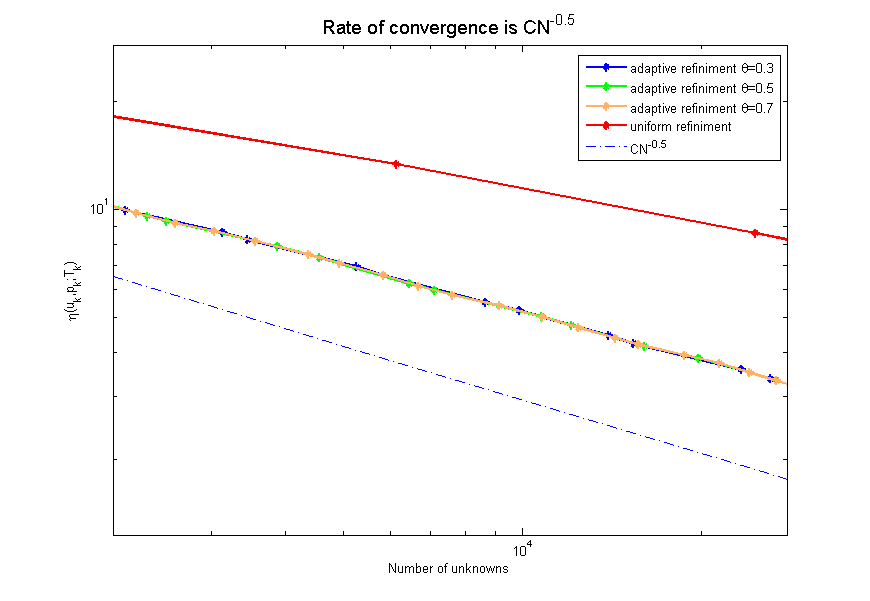}
 	\caption{Quasi optimality of the AMIPDG of the error $\eta\left(\boldsymbol{u}_{k}, {p}_{k} ; \mathcal{T}_{k}\right)$ with different marking parameters $\theta$.}\label{fig:ex2_res}
 \end{figure}
 
 % put your thanks here
 \section*{Acknowledgments}
 The authors are supported by the National Natural Science Foundation of China (Grant No. 12071160). The second author is also supported by the National Natural Science Foundation of China (Grant No. 11901212).

%%% 

\end{document}